\documentclass[11pt]{amsart}
\usepackage{fullpage}
\usepackage{color}
\usepackage{pstricks,pst-node,pst-plot} %This loads the pstricks package
\usepackage{graphicx,psfrag} %,epstopdf}
\usepackage{color} %This loads the color package
\usepackage{tikz}
\usepackage{pgffor}
\usepackage{hyperref}
\usepackage{todonotes}
\usepackage{subfigure}
\usepackage{verbatim}

\usepackage{bm}
\usepackage{multirow}

\usepackage{perpage}	 %the perpage package
\allowdisplaybreaks
	
\MakePerPage{footnote}	 %the perpage package command

\newtheorem{theorem}{Theorem}[section]
\newtheorem*{theorem-non}{Theorem}
\newtheorem{lemma}[theorem]{Lemma}

\newtheorem{proposition}[theorem]{Proposition}

\theoremstyle{definition}
\newtheorem{definition}[theorem]{Definition}
\newtheorem{example}[theorem]{Example}

\theoremstyle{remark}

\numberwithin{equation}{section}

\definecolor{gray}{rgb}{.5,.5,.5}

\definecolor{black}{rgb}{0,0,0}

\definecolor{blue}{rgb}{0,0,1}

\definecolor{red}{rgb}{1,0,0}

\definecolor{green}{rgb}{0,1,0}

\definecolor{yellow}{rgb}{1,1,.4}

\newrgbcolor{purple}{.5 0 .5}
\newrgbcolor{black}{0 0 0}
\newrgbcolor{white}{1 1 1}
\newrgbcolor{gold}{.5 .5 .2}
\newrgbcolor{darkgreen}{0 .5 0}
\newrgbcolor{gray}{.5 .5 .5}
\newrgbcolor{lightgray}{.75 .75 .75}
\newrgbcolor{lightred}{.75 0 0}

\newcommand{\crossneg}{
\begin{tikzpicture}[baseline=-2]
\draw[white,line width=1.5pt,double=black,double distance=.5pt] (0,-0.1) -- (0.3,0.2);
\draw[white,line width=1.5pt,double=black,double distance=.5pt] (0,0.2) -- (0.3,-0.1);
\end{tikzpicture}}

\newcommand{\crosspos}{
\begin{tikzpicture}[baseline=-2]
\draw[white,line width=1.5pt,double=black,double distance=.5pt] (0,0.2) -- (0.3,-0.1);
\draw[white,line width=1.5pt,double=black,double distance=.5pt] (0,-0.1) -- (0.3,0.2);
\end{tikzpicture}}

\begin{document}

\title{The average genus of a 2-bridge knot is asymptotically linear}

\author{Moshe Cohen}
\address{Mathematics Department, State University of New York at New Paltz, New Paltz, NY 12561}
\email{cohenm@newpaltz.edu}

\author{Adam M. Lowrance}
\address{Department of Mathematics and Statistics, Vassar College, Poughkeepsie, NY 12604}
\email{adlowrance@vassar.edu}
\thanks{The second author was supported by NSF grant DMS-1811344.}

\begin{abstract}
Experimental work suggests that the Seifert genus of a knot grows linearly with respect to the crossing number of the knot. In this article, we use a billiard table model for $2$-bridge or rational knots to show that the average genus of a $2$-bridge knot with crossing number $c$ asymptotically approaches $c/4+1/12$.
\end{abstract}

\maketitle

\section{Introduction}

The Seifert genus $g(K)$ of a knot $K$ in $S^3$ is the minimum genus of any oriented surface embedded in $S^3$ whose boundary is the knot $K$. Dunfield et al. \cite{Dun:knots} presented experimental data that suggests the Seifert genus of a knot grows linearly with respect to crossing number. Using a billiard table model for $2$-bridge knots developed by Koseleff and Pecker \cite{KosPec3, KosPec4}, Cohen \cite{Coh:lower} gave a lower bound on the average genus of a $2$-bridge knot. 

In this paper, we compute the average genus $\overline{g}_c$ of $2$-bridge knots with crossing number $c$ and show that $\overline{g}_c$ is asymptotically linear with respect to $c$. Let $\mathcal{K}_c$ be the set of unoriented $2$-bridge knots with $c$ crossings where only one of a knot and its mirror image is in the set. For example $|\mathcal{K}_3|=1$  and contains one of the right-handed or left-handed trefoil. Define the average genus $\overline{g}_c$ by
\begin{equation}
\label{eq:avgenus}
\overline{g}_c = \frac{\sum_{K\in\mathcal{K}_c} g(K)}{|\mathcal{K}_c|}.
\end{equation} 
Since the genus of a knot and the genus of its mirror image are the same, $\overline{g}_c$ is independent of the choice of each knot or its mirror image as elements in $\mathcal{K}_c$.

\begin{theorem}
\label{thm:mainformula}
Let $c\geq 3$. The average genus $\overline{g}_c$ of a $2$-bridge knot with crossing number $c$ is 
\[\overline{g}_c  = \frac{c}{4} + \frac{1}{12} + \varepsilon(c),\]
where
\[\varepsilon (c) =  \begin{cases}
\displaystyle\frac{2^{\frac{c-4}{2}} - 4}{12(2^{c-3}+2^{\frac{c-4}{2}})} & \text{if } c\equiv 0\text{ mod }4,\\
\displaystyle \frac{1}{3\cdot 2^{\frac{c-3}{2}}} & \text{if } c\equiv 1\text{ mod }4,\\
\displaystyle \frac{2^{\frac{c-4}{2}}+3c-11}{12(2^{c-3}+2^{\frac{c-4}{2}}-1)}& \text{if } c\equiv 2\text{ mod }4, \text{ and}\\
\displaystyle \frac{2^{\frac{c+1}{2}}+11-3c}{12(2^{c-3}+2^{\frac{c-3}{2}}+1)} & \text{if } c\equiv 3\text{ mod }4.
\end{cases}\]
Since $\varepsilon(c)\to 0$ as $c\to \infty$, the average genus $\overline{g}_c$ approaches $\frac{c}{4}+\frac{1}{12}$ as $c \to \infty$.
\end{theorem}

Suzuki and Tran \cite{SuzukiTran} independently proved this formula for $\overline{g}_c$. Ray and Diao \cite{RayDiao} expressed $\overline{g}_c$ using sums of products of certain binomial coefficients.  Baader, Kjuchukova, Lewark, Misev, and Ray \cite{BKLMR} previously showed that if $c$ is sufficiently large, then $\frac{c}{4} \leq \overline{g}_c$. 

The proof of Theorem \ref{thm:mainformula} uses the Chebyshev billiard table model for knot diagrams of Koseleff and Pecker \cite{KosPec3,KosPec4} as presented by Cohen and Krishnan \cite{CoKr} and with Even-Zohar \cite{CoEZKr}. This model yields an explicit enumeration of the elements of $\mathcal{K}_c$ as well as an alternating diagram in the format of Figure \ref{fig:alternating} for each element of $\mathcal{K}_c$. Murasugi \cite{Mur:genus} and Crowell \cite{Cro:genus} proved that the genus of an alternating knot is the genus of the surface obtained by applying Seifert's algorithm \cite{Sei} to an alternating diagram of the knot. The proof of Theorem \ref{thm:mainformula} proceeds by applying Seifert's algorithm to the alternating diagrams obtained from our explicit enumeration of $\mathcal{K}_c$ and averaging the genera of those surfaces.

This paper is organized as follows. In Section \ref{sec:background}, we recall how the Chebyshev billiard table model for $2$-bridge knots diagrams can be used to describe the set $\mathcal{K}_c$ of $2$-bridge knots. In Section \ref{sec:recursions}, we find recursive formulas that allow us to count the total number of Seifert circles among all $2$-bridge knots with crossing number $c$. Finally in Section \ref{sec:formulas}, we find a closed formula for the number of Seifert circles among all $2$-bridge knots and use that to prove Theorem \ref{thm:mainformula}.

\section{Background}
\label{sec:background}

The average genus of $2$-bridge knots with crossing number $c$ is the quotient of the sum of the genera of all $2$-bridge knots with crossing number $c$ and the number of $2$-bridge knots with crossing number $c$. Ernst and Sumners \cite{ErnSum} proved formulas for the number $|\mathcal{K}_c|$ of $2$-bridge knots.

\begin{theorem}[Ernst-Sumners \cite{ErnSum}, Theorem 5]
\label{thm:ernstsumners}
The number $|\mathcal{K}_c|$ of 2-bridge knots with $c$ crossings where chiral pairs are \emph{not} counted separately is given by
\[
|\mathcal{K}_c| = 
\begin{cases}
\frac{1}{3}(2^{c-3}+2^{\frac{c-4}{2}}) & \text{ for }4 \geq c\equiv 0 \text{ mod }4,\\
\frac{1}{3}(2^{c-3}+2^{\frac{c-3}{2}}) & \text{ for }5\geq c\equiv 1 \text{ mod }4, \\
\frac{1}{3}(2^{c-3}+2^{\frac{c-4}{2}}-1) & \text{ for }6 \geq c\equiv 2 \text{ mod }4, \text{ and}\\
\frac{1}{3}(2^{c-3}+2^{\frac{c-3}{2}}+1) & \text{ for }3\geq c\equiv 3 \text{ mod }4.
\end{cases}
\]
\end{theorem}

A billiard table diagram of a knot is constructed as follows. Let $a$ and $b$ be relatively prime positive integers with $a<b$, and consider an $a\times b$ grid. Draw a sequence of line segments along diagonals of the grid as follows. Start at the bottom left corner of the grid with a line segment that bisects the right angle of the grid. Extend that line segment until it reaches an outer edge of the grid, and then start a new segment that is reflected $90^\circ$. Continue in this fashion until a line segment ends in a corner of the grid. Connecting the beginning of the first line segment with the end of the last line segment results in a piecewise linear closed curve in the plane with only double-point self-intersections. If each such double-point self-intersection is replaced by a crossing, then one obtains a \emph{billiard table diagram} of a knot. See Figure \ref{fig:billiard}.

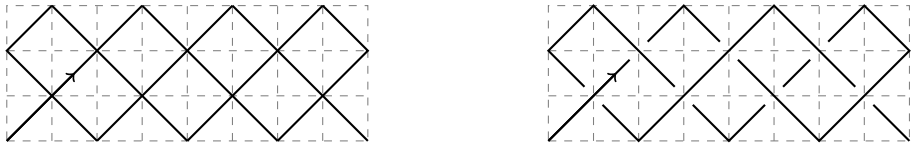
\begin{figure}[h]
\begin{tikzpicture}[scale=.6]
% Projection
% Dashed Grid
\draw[dashed, white!50!black] (0,0) rectangle (8,3);
\foreach \x in {1,...,7}
	{\draw[dashed, white!50!black] (\x,0) -- (\x,3);}
\foreach \x in {1,2}
	{\draw[dashed, white!50!black] (0,\x) -- (8, \x);}
%Billiard projection
\foreach \x in {0,2,4}
	{\draw[thick] (\x,0) -- (\x+3,3);
	\draw[thick] (\x+1,3) -- (\x+4,0);}
\draw[thick] (1,3) -- (0,2) -- (2,0);
\draw[thick] (6,0) -- (8,2) -- (7,3);
\draw[thick, ->] (0,0) -- (1.5,1.5);

% Knot Diagram
\begin{scope}[xshift = 12 cm]
	% Dashed Grid
	\draw[dashed, white!50!black] (0,0) rectangle (8,3);
	\foreach \x in {1,...,7}
		{\draw[dashed, white!50!black] (\x,0) -- (\x,3);}
	\foreach \x in {1,2}
		{\draw[dashed, white!50!black] (0,\x) -- (8, \x);}
	% Knot Projection
	\draw[thick] (0,0) -- (1.8,1.8);
	\draw[thick] (2.2, 2.2) -- (3,3) -- (3.8,2.2);
	\draw[thick] (4.2,1.8) -- (6,0) -- (8,2) -- (7,3) -- (6.2,2.2);
	\draw[thick] (5.8,1.8) -- (5.2,1.2);
	\draw[thick] (4.8,0.8) -- (4,0) -- (3.2,0.8);
	\draw[thick] (2.8,1.2) -- (1,3) -- (0,2) -- (0.8,1.2);
	\draw[thick] (1.2,0.8) -- (2,0) -- (5,3) -- (6.8,1.2);
	\draw[thick] (7.2, 0.8) -- (8,0);	
	\draw[thick, ->] (0,0) -- (1.5,1.5);	
\end{scope}		
\end{tikzpicture}
\caption{A billiard table projection and a billiard table diagram of a knot on a $3\times 8$ grid. The diagram corresponds to the word $+-++ -{}-+$. We do not draw the arc connecting the ends but understand it to be present.}
\label{fig:billiard}
\end{figure}

Billiard table diagrams on a $3\times b$ grid have bridge number either one or two, that is, such a knot is either the unknot or a $2$-bridge knot. In a $3\times b$ billiard table diagram, there is one crossing on each vertical grid line except the first and the last. A string of length $b-1$ in the symbols $\{+,-\}$ determines a $2$-bridge knot or the unknot, as follows. A crossing corresponding to a $+$ looks like $\tikz[baseline=.6ex, scale = .4]{
\draw (0,0) -- (1,1);
\draw (0,1) -- (.3,.7);
\draw (.7,.3) -- (1,0);
}
~$, and a crossing corresponding to a $-$ looks like $\tikz[baseline=.6ex, scale = .4]{
\draw (0,0) -- (.3,.3);
\draw (.7,.7) -- (1,1);
\draw (0,1) -- (1,0);
}
~$. Figure \ref{fig:billiard} shows an example.

A given $2$-bridge knot has infinitely many descriptions as strings of various lengths in the symbols $\{+,-\}$. Cohen, Krishnan, and Evan-Zohar's work \cite{CoKr, CoEZKr} lets us describe $2$-bridge knots in this manner but with more control on the number of strings representing a given $2$-bridge knot. 
\begin{definition}
Define the \emph{partially double-counted set $T(c)$ of $2$-bridge words with crossing number $c$} as follows. Each word in $T(c)$ is a word in the symbols $\{+,-\}$. If $c$ is odd, then a word $w$ is in $T(c)$ if and only if it is of the form
\[
(+)^{\varepsilon_1}(-)^{\varepsilon_2}(+)^{\varepsilon_3}(-)^{\varepsilon_4}\ldots(-)^{\varepsilon_{c-1}}(+)^{\varepsilon_c}, \]
where $\varepsilon_i\in\{1,2\}$ for $i\in\{1,\ldots,c\}$, $\varepsilon_1=\varepsilon_c=1$, and the length of the word $\ell=\sum_{i=1}^{c}\varepsilon_i \equiv 1$ mod $3$. Similarly, if $c$ is even, then a word $w$ is in $T(c)$ if and only if it is of the form
\[(+)^{\varepsilon_1}(-)^{\varepsilon_2}(+)^{\varepsilon_3}(-)^{\varepsilon_4}\ldots(+)^{\varepsilon_{c-1}}(-)^{\varepsilon_c},\]
where $\varepsilon_i\in\{1,2\}$ for $i\in\{1,\ldots,c\}$, $\varepsilon_1=\varepsilon_c=1$, and the length of the word $\ell=\sum_{i=1}^{c}\varepsilon_i \equiv 1$ mod $3$.
\end{definition}
The set $T(c)$ is described as partially double-counted because every $2$-bridge knot is represented by exactly one or two words in $T(c)$, as described in Theorem \ref{thm:list} below. Although the billiard table diagram associated with $w$ has $\ell$ crossings, there is an alternating diagram associated with $w$ that has $c$ crossings, and hence we use the $T(c)$ notation.

The \emph{reverse} $r(w)$ of a word $w$ of length $\ell$ is a word whose $i$th entry is the $(\ell - i +1)$st entry of $w$; in other words, $r(w)$ is just $w$ backwards. The \emph{reverse mirror} $\overline{r}(w)$ of a word $w$ of length $\ell$ is the word of length $\ell$ where each entry disagrees with the corresponding entry of $r(w)$; in other words, $\overline{r}(w)$ is obtained from $w$ by reversing the order and then changing every $+$ to a $-$ and vice versa. 
\begin{definition}
The subset $T_p(c)\subset T(c)$ of \emph{words of palindromic type} consists of words $w\in T(c)$ such that $w=r(w)$ when $c$ is odd and $w=\overline{r}(w)$ when $c$ is even.
\end{definition}
\noindent For example, the word $w=+ -{}-+$ is the only word in $T_p(3)$, and the word $w=+ - + -$ is the only word in $T_p(4)$. 

The following theorem says exactly which $2$-bridge knots are represented by two words in $T(c)$ and which $2$-bridge knots are represented by only one word in $T(c)$. The theorem is based on work by Schubert \cite{Sch} and  Koseleff and Pecker \cite{KosPec4}. The version of the theorem we state below comes from Lemma 2.1 and Assumption 2.2 in \cite{Coh:lower}.

\begin{theorem}
\label{thm:list}
Let $c\geq 3$. Every $2$-bridge knot is represented by a word in $T(c)$. If a $2$-bridge knot $K$ is represented by a word $w$ of palindromic type, that is, a word in $T_p(c)$, then $w$ is the only word in $T(c)$ that represents $K$. If a $2$-bridge knot $K$ is represented by a word $w$ that is not in $T_p(c)$, then there are exactly two words in $T(c)$ that represent $K$, namely $w$ and $r(w)$ when $c$ is odd or $w$ and $\overline{r}(w)$ when $c$ is even.
\end{theorem}

A billiard table diagram associated with a word $w$ in $T(c)$ is not necessarily alternating; however the billiard table diagram associated with $w$ can be transformed into an alternating diagram $D$ of the same knot as follows. A \emph{run} in $w$ is a subword of $w$ consisting of all the same symbols (either all $+$ or all $-$) that is not properly contained in a single-symbol subword of longer length. By construction, if $w\in T(c)$, then it is made up of $c$ runs all of length one or two. The run $+$ is replaced by $\sigma_1$, the run $++$ is replaced by $\sigma_2^{-1}$, the run $-$ is replaced by $\sigma_2^{-1}$ and the run $-{}-$ is replaced by $\sigma_1$, as summarized by pictures in Table \ref{tab:wtoD}.

The left side of the diagram has a strand entering from the bottom left and a cap on the top left. If the last term is $\sigma_1$, then the right side of the diagram has a strand exiting to the bottom right and a cap to the top right, and if the last term is $\sigma_2^{-1}$, then the right side of the diagram has a strand exiting to the top right and a cap on the bottom right. See Figure \ref{fig:alternating} for an example. Theorem 2.4 and its proof in \cite{Coh:lower} explain this correspondence.

\begin{center}
\begin{table}[h]
\begin{tabular}{|c||c|c|c|c|}
\hline
&&&&\\
Run in billiard table diagram word $w$ 				& $(+)^1$			& $(+)^2$					& $(-)^1$					& $(-)^2$			\\
&&&&\\
\hline
&&&&\\
Crossing in alternating diagram $D$		& $\sigma_1$	& $\sigma_2^{-1}$	& $\sigma_2^{-1}$	& $\sigma_1$	\\
&&&&\\
&& $\crossneg$ & $\crossneg$ &\\
&$\crosspos$ &&& $\crosspos$ \\
&&&&\\
\hline
\end{tabular}
\caption{Transforming a billiard table diagram into an alternating diagram, as seen in \cite[Table 1]{Coh:lower}.}
\label{tab:wtoD}
\end{table}
\end{center}

\begin{figure}[h]
\begin{tikzpicture}[scale=.6]
% Knot Diagram

% Dashed Grid
\draw[dashed, white!50!black] (0,0) rectangle (8,3);
\foreach \x in {1,...,7}
	{\draw[dashed, white!50!black] (\x,0) -- (\x,3);}
\foreach \x in {1,2}
	{\draw[dashed, white!50!black] (0,\x) -- (8, \x);}
% Knot Diagram
\draw[thick] (0,0) -- (1.8,1.8);
\draw[thick] (2.2, 2.2) -- (3,3) -- (3.8,2.2);
\draw[thick] (4.2,1.8) -- (6,0) -- (8,2) -- (7,3) -- (6.2,2.2);
\draw[thick] (5.8,1.8) -- (5.2,1.2);
\draw[thick] (4.8,0.8) -- (4,0) -- (3.2,0.8);
\draw[thick] (2.8,1.2) -- (1,3) -- (0,2) -- (0.8,1.2);
\draw[thick] (1.2,0.8) -- (2,0) -- (5,3) -- (6.8,1.2);
\draw[thick] (7.2, 0.8) -- (8,0);	
\draw[thick, ->] (0,0) -- (1.5,1.5);			

% Alternating Diagram
\begin{scope}[xshift=12cm, thick, rounded corners = 2mm]
	\draw[->] (0,0) -- (1.5,1.5);
	\draw (0,0) -- (1.8,1.8);
	\draw (2.2,2.2) -- (3,3) -- (4.8,1.2);
	\draw (5.2,0.8) -- (6,0) -- (8,2) -- (7,3) -- (5,3) -- (4.2,2.2);
%	\draw (8,2) arc (-45:90:0.5857864);
%	\draw (7.58578643763, 3) -- (5,3) -- (4.2,2.2);
	\draw (3.8,1.8) -- (3,1) -- (1,3) -- (0,2) -- (0.8,1.2);
	\draw (1.2,0.8) -- (2,0) -- (4,0) -- (6,2) -- (6.8,1.2);
	\draw (7.2,0.8) -- (8,0);
\end{scope}
\end{tikzpicture}
\caption{The billiard table diagram knot corresponding to the word $+-++ -{}-+$ has alternating diagram $\sigma_1\sigma_2^{-2}\sigma_1^2$. }
\label{fig:alternating}
\end{figure}
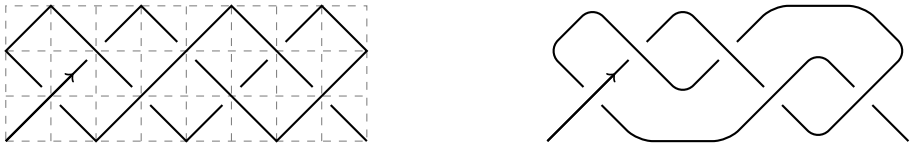

Murasugi \cite{Mur:genus} and Crowell \cite{Cro:genus} proved that the genus of an alternating knot $K$ is the genus of the Seifert surface obtained from Seifert's algorithm on an alternating diagram of $K$. Therefore, the average genus $\overline{g}_c$ is
\[ \overline{g}_c = \frac{1}{2}\left(1 + c - \overline{s}_c \right),\]
where $\overline{s}_c$ is the average number of Seifert circles in the alternating diagrams of all $2$-bridge knots with crossing number $c$. In Section \ref{sec:recursions}, we find recursive formulas for the total number of Seifert circles in the alternating diagrams associated with words in $T(c)$ and $T_p(c)$, named $s(c)$ and $s_p(c)$, respectively. Theorem \ref{thm:list} implies that
\begin{equation}
\label{eq:avseifert}
 \overline{s}_c = \frac{s(c) + s_p(c)}{2|\mathcal{K}_c|}.
 \end{equation}

Seifert's algorithm uses the orientation of a knot diagram to construct a Seifert surface. Lemma 3.3 in \cite{Coh:lower} keeps track of the orientations of the crossings in the alternating diagram $D$ associated with a word $w$ in $T(c)$.  See also Property 7.1 in \cite{Co:3-bridge}.

\begin{lemma}
\label{lem:or1}
\cite[Lemma 3.3]{Coh:lower}
The following conventions determine the orientation of every crossing in the alternating diagram $D$ associated with a word $w$ in $T(c)$.
\begin{enumerate}
\item Two of the three strands in $D$ are oriented to the right.
\item If either a single $+$ or a single $-$ appears in a position congruent to $1$ modulo $3$ in $w$, then it corresponds to a single crossing in the alternating diagram $D$ that is horizontally-oriented.
\item If either a double $++$ or a double $-{}-$ appears in two positions congruent to $2$ and $3$ modulo $3$ in $w$, then they correspond to a single crossing in the alternating diagram $D$ that is horizontally-oriented.
\item The remaining crossings in $D$ are vertically-oriented.
\end{enumerate}
\end{lemma}

\section{Recursive formulas for Seifert circles}
\label{sec:recursions}

In this section, we find recursive formulas for the total number of Seifert circles in the alternating diagrams associated with words in $T(c)$ and $T_p(c)$. The section is split between the general case, where we deal with $T(c)$, and the palindromic case, where we deal with $T_p(c)$.

\subsection{General case}
\label{subsec:general}

In order to develop the recursive formulas for the total number of Seifert circles of alternating diagrams coming from $T(c)$, we partition $T(c)$ into four subsets. The final run of each of word $w$ in $T(c)$ is fixed by construction; if $c$ is odd, then $w$ ends in a single $+$, and if $c$ is even, then $w$ ends in a single $-$.  Suppose below that $c$ is odd; the even case is similar. 

The two penultimate runs in a word in $T(c)$ must be exactly one of the following cases:
\begin{itemize}
	\item[(1)] a single + followed by a single -,
	\item[(2)] a double ++ followed by a double -{}-,
	\item[(3)] a single + followed by a double -{}-, or
	\item[(4)] a double ++ followed by a single -.
\end{itemize}
These four cases form a partition of $T(c)$.

The Jacobsthal sequence \href{https://oeis.org/A001045}{A001045} \cite{OEIS1045} is an integer sequence satisfying the recurrence relation $J(n) = J(n-1) + 2J(n-2)$ with initial values $J(0)=0$ and $J(1)=1$. The closed formula for the $n$th Jacobsthal number  is $J(n)=\frac{2^n - (-1)^n}{3}$. We use the Jacobsthal sequence to find a formula for the number of words in $T(c)$.
\begin{proposition}
\label{prop:countterms}
The number $t(c) = \frac{2^{c-2} - (-1)^c}{3}$ is the Jacobsthal number $J(c-2)$ and satisfies the recursive formula $t(c)=t(c-1)+2t(c-2)$.
\end{proposition}

\begin{proof}
The base cases of $t(3)=t(4)=1$ hold because $T(3) =\{+-{}-+\}$ and $T(4) = \{+-+-\}$.

Next, we show that $t(c)$ satisfies the recursive formula above.  
The penultimate two runs in cases 3 and 4 are of length three, which is convenient for our model, and so they can be removed without changing the length requirement modulo 3. Removing either $+-{}-$ or $++-$ also does not affect the parity of the number of crossings. The final $+$ after these subwords can still be appended to the shorter words after the removal. What is left after removal in each of these cases is the set $T(c-2)$, and so cases 3 and 4 combine to contribute $2t(c-2)$ words.

In case 1, the final three runs $+-+$ can be replaced by $++-$, preserving the length of the word and reducing the number of crossings by one. In case 2, the final three runs $++-{}-+$ can be replaced by $+-$ without changing the length requirement modulo 3. In this case, the number of crossings is reduced by one. These two cases partition $T(c-1)$.  In case 1, the penultimate run is a double, and in case 2, it is a single.  Thus these two cases together contribute $t(c-1)$ words. 

Therefore $t(c) = t(c-1) + 2t(c-2)$. Since $t$ satisfies the Jacobsthal recurrence relation and $t(3)=t(4)=J(1)=J(2)=1$, it follows that $t(c) = J(c-2)=  \frac{2^{c-2} - (-1)^c}{3}$.
\end{proof}

The replacements in the proof of Proposition \ref{prop:countterms} can be summarized as follows.
\begin{itemize}
	\item[(1)] The final string $+-+$ is replaced by $++-$, obtaining a new word with $c-1$ crossings.
	\item[(2)] The final string $++-{}-+$ is replaced by $+-$, obtaining a new word with $c-1$ crossings.
	\item[(3)] The final string $+-{}-+$ is replaced by $+$, obtaining a new word with $c-2$ crossings.
	\item[(4)] The final string $++-+$ is replaced by $+$, obtaining a new word with $c-2$ crossings.
\end{itemize}

\begin{example}

\label{ex:c6countterms}
Table \ref{tab:c456} shows the sets $T(4)$, $T(5)$, and $T(6)$. Subwords of words in $T(6)$ in parentheses are replaced according to the proof of Proposition \ref{prop:countterms} to obtain the words on the left in either $T(4)$ or $T(5)$. We see that $t(6) = t(5) + 2t(4)$.
\end{example}

\begin{center}
\begin{table}[h]
\begin{tabular}{|c|c||c|c|}
\hline
$T(4)$ & $+-+()-$ 				& $+-+(-++)-$ 				& \\
\cline{1-2}
$T(4)$ & $+-+()-$ 				& $+-+(-{}-+)-$ 			& \\
\cline{1-2}
\multirow{3}{*}{$T(5)$}			& $+-{}-++(-)+$ 	& $+-{}-++(-{}-++)-$ 	& $T(6)$\\
														& $+-++(-{}-)+$ 	& $+-++(-+)-$ 				& \\
														& $+-{}-+(-{}-)+$ & $+-{}-+(-+)-$ 		& \\
\hline
\end{tabular}
\caption{The sets $T(4)$, $T(5)$, and $T(6)$ with the subwords in the parentheses replaced as in the proof of Proposition \ref{prop:countterms}.}
\label{tab:c456}
\end{table}
\end{center}

\begin{example}
\label{ex:c7countterms}
Table \ref{tab:c567} shows the sets $T(5)$, $T(6)$, and $T(7)$. Subwords of words in $T(7)$ in parentheses are replaced according to the proof of Proposition \ref{prop:countterms} to obtain the words on the left in either $T(5)$ or $T(6)$. We see that $t(7) = t(6) + 2t(5)$.
\end{example}

\begin{center}
\begin{table}[h]
\begin{tabular}{|c|c||c|c|}
\hline
			& $+-{}-++-()+$ 			& $+-{}-++-(+--)+$ 	& \\
$T(5)$ & $+-++-{}-()+$ 			& $+-++-{}-(+--)+$ 				& \\
			& $+-{}-+-{}-()+$ 		& $+-{}-+-{}-(+--)+$ 		& \\
\cline{1-2}
			& $+-{}-++-()+$ 			& $+-{}-++-(++-)+$ 	& \\
$T(5)$ & $+-++-{}-()+$ 			& $+-++-{}-(++-)+$ 				& \\
			& $+-{}-+-{}-()+$ 		& $+-{}-+-{}-(++-)+$ 		& $T(7)$ \\
\cline{1-2}
			& $+-+-{}-(+)-$ 			& $+-+-{}-(++--)+$ 			& \\
			& $+-++-(+)-$ 				& $+-++-(++--)+$ 				& \\
$T(6)$	& $+-{}-+-(+)-$ 			& $+-{}-+-(++--)+$ 			& \\
			& $+-+-(++)-$ 				& $+-+-(+-)+$ 				& \\ 
			& $+-{}-++-{}-(++)-$ 	& $+-{}-++-{}-(+-)+$ 	& \\
\hline
\end{tabular}
\caption{The sets $T(5)$, $T(6)$, and $T(7)$ with the subwords in the parentheses replaced as in the proof of Proposition \ref{prop:countterms}.}
\label{tab:c567}
\end{table}
\end{center}

Let $s(c)$ be the total number of Seifert circles obtained when Seifert's algorithm is applied to the alternating diagrams associated to words in $T(c)$. For brevity, we say that $s(c)$ is the total number of Seifert circles from $T(c)$. In order to find a recursive formula for $s(c)$, we develop recursive formulas for sizes of the subsets in the partition of $T(c)$ defined by the four cases above.

\begin{lemma}
\label{lem:countcases}
Let $t_1(c)$, $t_2(c)$, $t_3(c)$, and $t_4(c)$ be the number of words in cases 1, 2, 3, and 4, respectively, for crossing number $c$. Then 
\[t_1(c)=2t(c-3),~t_2(c)=t(c-2),~\text{and}~t_3(c)=t_4(c)=t(c-2).\]
\end{lemma}

\begin{proof}
The last result $t_3(c)=t_4(c)=t(c-2)$ appears in the proof of Proposition \ref{prop:countterms} above.  We now consider the other cases.

Without loss of generality, suppose $c$ is odd.  In case 2, the final three runs are $++-{}-+$, and we can obtain a word with crossing number $c-1$ by replacing this string with $+-$, as described in Proposition \ref{prop:countterms} above. If the $(c-3)$rd run is a double $-{}-$, then the string $-{}-++-{}-$ in positions $c-3$ through $c-1$ can be removed without affecting the required length modulo 3, with the final single $+$ becoming a final single $-$.  The number of such words is $t(c-3)$. If the $(c-3)$rd run is a single $-$, then $-++-{}-+$ is replaced with the string $-+-$.  This is case 1 for $c-1$ crossings, and so the number of these words is $t_1(c-1)$. Therefore $t_2(c) = t(c-3)+t_1(c-1)$.

In case 1, the final three runs are $+-+$ and we can reduce this to a word with crossing number $c-1$ by replacing this string with $++-$, as described in Proposition \ref{prop:countterms} above. If the $(c-3)$rd run is a single $-$, then first perform the replacement move, yielding the string $-++-$, and then remove the penultimate two runs without affecting the required length modulo 3, keeping the final single $-$.  The number of these words is $t(c-3)$. If the $(c-3)$rd run is a double $-{}-$, then after performing the replacement move, the final three runs are $-{}-++-$.  This is case 2 for $c-1$ crossings, and so the number of these words is $t_2(c-1)$. Therefore $t_1(c)=t(c-3)+t_2(c-1)$.

We prove that $t_1(c)=2t(c-3)$ and that $t_2(c)=t(c-2)$ by induction. For the base cases, Example \ref{ex:c6countterms} implies that $t_2(5)=1$ and $t_1(6)=2$, and $t(3)=1$ because $T(3)=\{+--+\}$.  Our inductive hypothesis is that $t_1(c-1)=2t(c-4)$ and $t_2(c-1)=t(c-3)$. We then have that 
\[t_1(c) = t(c-3) + t_2(c-1) = 2t(c-3)\]
and
\[t_2(c)=t(c-3)+t_1(c-1) = t(c-3) + 2t(c-4) = t(c-2).\]
\end{proof}

We are now ready to prove our recursive formula for $s(c)$, the total number of Seifert circles from $T(c)$. Throughout the proof, we refer to Table \ref{tab:Seifert} below.

\begin{table}[h]
\begin{tabular}{|c|c||c|c|c|}
\hline
Case & Crossing & String & Alternating  & Seifert State \\
 & Number & & Diagram& \\
\hline
\hline
1 & $c$ & $+-+$ & 
% Case 1 knot (before)
\begin{tikzpicture}[scale=.5, rounded corners = 1mm]
\draw[white] (-.2,-.2) rectangle (3.2,2.2);
\draw (0,0) -- (1.3, 1.3);
\draw (0,1) -- (.3,.7);
\draw (.7,.3) -- (1,0) -- (2,0) -- (3,1) -- (2,2) -- (1.7,1.7);
\draw (0,2) -- (1,2) -- (2.3,.7);
\draw (2.7,.3) -- (3,0);

\draw[->] (.5, .5) -- (.1,.1);
\draw[->] (.7,.3) -- (.9,.1);
\draw[->] (2.5, .5) -- (2.9,.9);
\draw[->] (2.7,.3) -- (2.9,.1);
\draw[->] (1.5, 1.5) -- (1.9,1.1);
\draw[->] (1.3,1.3) -- (1.1,1.1);

\end{tikzpicture}
&
% Case 1 Seifert state before
\begin{tikzpicture}[scale=.5, rounded corners = 1mm]
\draw[white] (-.2,-.2) rectangle (3.2,2.2);
\draw[->] (0,1) -- (.4,.5) -- (0,0);
\draw[->] (0,2) -- (1,2) -- (1.4,1.5) -- (.6,.5) -- (1,0) -- (2,0) -- (2.5,.4) -- (3,0);
\draw[->] (2,1) -- (2.5,.6) -- (3,1) -- (2,2) -- (1.6,1.5) -- (2,1);

\end{tikzpicture}
\\
\hline
1 & $c-1$ & $++-$ & 
% Case 1 knot after
\begin{tikzpicture}[scale = .5, rounded corners = 1mm]
\draw[white] (-.2,-.2) rectangle (2.2,2.2);
\draw (0,0) -- (1,0) -- (2,1) -- (1.7,1.3);
\draw (1.3,1.7) -- (1,2) -- (0,1);
\draw (0,2) -- (0.3,1.7);
\draw (.7,1.3) -- (1,1) -- (2,2);
\draw[->] (0.5,1.5) -- (.9,1.9);
\draw[->] (.7,1.3) -- (.9,1.1);
\draw[->] (1.5,1.5) -- (1.9,1.9);
\draw[->] (1.7, 1.3) -- (1.9,1.1);
\end{tikzpicture}
&
% Case 1 Seifert state after
\begin{tikzpicture}[scale = .5, rounded corners = 1mm]
\draw[white] (-.2,-.2) rectangle (2.2,2.2);
\draw[->] (0,2) -- (.5,1.6) -- (1,2) -- (1.5,1.6) -- (2,2);
\draw[->] (0,1) -- (.5, 1.4) -- (1,1) -- (1.5,1.4) -- (2,1) -- (1,0) -- (0,0);
\end{tikzpicture}
 \\
 \hline\hline
 2A & $c$ & $-++-{}-+$ & 
% Case 2A knot before
\begin{tikzpicture}[scale = .5, rounded corners = 1mm]
\draw[white] (-1.2,-.2) rectangle (3.2,2.2);
\draw (-1,0) -- (1,0) -- (2,1) -- (2.3,.7);
\draw (2.7,.3) -- (3,0);
\draw (-1,2) -- (0,1) -- (.3,1.3);
\draw (-.3,1.7) -- (0,2) -- (1.3,.7);
\draw (-1,1) -- (-.7,1.3);
\draw (1.7,.3) -- (2,0) -- (3,1) -- (2,2) -- (1,2) -- (.7,1.7);
\draw[->] (.3,1.3) -- (.1,1.1);
\draw[->] (.5,1.5) -- (.9,1.1);
\draw[->] (1.5,.5) -- (1.9,.9);
\draw[->] (1.7,.3) -- (1.9,.1);
\draw[->] (2.5,.5) -- (2.9,.9);
\draw[->] (2.7,.3) -- (2.9,.1);
\draw[->] (-.5,1.5) -- (-.9,1.9);
\draw[->] (-.3,1.7) -- (-.1,1.9);
\end{tikzpicture}
&
 % Case 2A Seifert state before
 \begin{tikzpicture}[scale = .5, rounded corners = 1mm]
 \draw[white] (-1.2,-.2) rectangle (3.2,2.2);
 \draw[->] (0,2) arc (90:-270:.4cm and .5cm);
 \draw[->] (-1,0) -- (1,0) -- (1.5,.4) -- (2,0) -- (2.5,.4) -- (3,0);
 \draw[->] (1.5,2) -- (1,2) -- (.6,1.5) -- (1,1) -- (1.5,.6) -- (2,1) -- (2.5,.6) -- (3,1) -- (2,2) -- (1.5,2);
 \draw[->] (-1,1) -- (-.6,1.5) -- (-1,2);
 \end{tikzpicture}
\\
\hline
2A & $c-1$ & $-+-$ & 
% Case 2A knot after
\begin{tikzpicture}
[scale = .4, rounded corners = 1mm]
\draw[white] (-1.2,-.2) rectangle (2.2,2.2);
\draw (-1,0) -- (0,0) -- (1.3,1.3);
\draw (1.7,1.7)--(2,2);
\draw (-1,2) -- (0.3,0.7);
\draw (0.7,0.3) -- (1,0) -- (2,1) -- (1,2) -- (0,2) -- (-.3,1.7);
\draw (-1,1) -- (-.7,1.3);

\draw[->] (-.3,1.7) -- (-.1,1.9);
\draw[->] (-.5,1.5) -- (-.9,1.9);
\draw[->] (0.5,0.5) -- (0.9, 0.9);
\draw[->] (0.3,0.7) -- (0.1,0.9);
\draw[->] (1.5,1.5) -- (1.9,1.1);
\draw[->] (1.7, 1.7) -- (1.9, 1.9);

\end{tikzpicture}

&
% Case 2A Seifert state after
\begin{tikzpicture}
[scale = .4, rounded corners = 1mm]
\draw[white] (-1.2,-.2) rectangle (2.2,2.2);
\draw[->] (-1,0) --(0,0) -- (.4,.5) -- (0,1) -- (-.4,1.5) -- (0,2)-- (1,2) --(1.5,1.6) -- (2,2);
\draw[->] (1,1) -- (1.5,1.4) -- (2,1) -- (1,0) -- (0.6,0.5) -- (1,1);
\draw[->] (-1,1) -- (-.6,1.5) -- (-1,2);
\end{tikzpicture}
 \\
\hline
\hline
2B & $c$ & $-{}-++-{}-+$ &
% Case 2B knot before
\begin{tikzpicture}[scale = .5, rounded corners = 1mm]
\draw[white] (-1.2,-.2) rectangle (3.2,2.2);
\draw (-.3,.3) -- (0,0) -- (1,0) -- (2,1) -- (2.3,.7);
\draw (2.7,.3) -- (3,0);
\draw (-1,0) -- (.3,1.3);
\draw (-1,2) -- (0,2) -- (1.3,.7);
\draw (1.7,.3) -- (2,0) -- (3,1) -- (2,2) -- (1,2) -- (.7,1.7);
\draw (-1,1) -- (-.7,.7);
\draw[->] (.3,1.3) -- (.1,1.1);
\draw[->] (.5,1.5) -- (.9,1.1);
\draw[->] (1.5,.5) -- (1.9,.9);
\draw[->] (1.7,.3) -- (1.9,.1);
\draw[->] (2.5,.5) -- (2.9,.9);
\draw[->] (2.7,.3) -- (2.9,.1);
\draw[->] (-.5,.5) -- (-1,0);
\draw[->] (-.3,.3) -- (-.1,.1);
\end{tikzpicture}
 & 
 % Case 2B Seifert state before
 \begin{tikzpicture}[scale = .5, rounded corners = 1mm]
 \draw[white] (-1.2,-.2) rectangle (3.2,2.2);
 \draw[->] (-1,2) -- (0,2) -- (.4,1.5) -- (0,1) -- (-.4,.5) -- (0,0) -- (1,0) -- (1.5,.4) -- (2,0) -- (2.5,.4) -- (3,0);
 \draw[->] (1.5,2) -- (1,2) -- (.6,1.5) -- (1,1) -- (1.5,.6) -- (2,1) -- (2.5,.6) -- (3,1) -- (2,2) -- (1.5,2);
 \draw[->] (-1,1) -- (-.6,.5) -- (-1,0);
 \end{tikzpicture}
 \\
\hline
2B & $c-1$ & $-{}-+-$ &
% Case 2B knot after
\begin{tikzpicture}
[scale = .4, rounded corners = 1mm]
\draw[white] (-1.2,-.2) rectangle (2.2,2.2);
\draw (-.3,.3) -- (0,0) -- (1.3,1.3);
\draw (1.7,1.7)--(2,2);
\draw (-1,0) -- (0,1) -- (0.3,0.7);
\draw (-1,1) -- (-.7,.7);
\draw (0.7,0.3) -- (1,0) -- (2,1) -- (1,2) -- (0,2) -- (-1,2);

\draw[->] (0.5,0.5) -- (0.9, 0.9);
\draw[->] (0.3,0.7) -- (0.1,0.9);
\draw[->] (1.5,1.5) -- (1.9,1.1);
\draw[->] (1.7, 1.7) -- (1.9, 1.9);
\draw[->] (-.5,.5) -- (-.9,.1);
\draw[->] (-.3,.3) -- (-.1,.1);

\end{tikzpicture}
 & % Case 2B Seifert state after
\begin{tikzpicture}
[scale = .4, rounded corners = 1mm]
\draw[white] (-1.2,-.2) rectangle (2.2,2.2);
\draw[->] (0,1) arc (90:450:.4cm and .5cm);
\draw[->] (-1,1) -- (-.6,.5) -- (-1,0);
\draw[->] (-1,2) -- (1,2) --(1.5,1.6) -- (2,2);
\draw[->] (1,1) -- (1.5,1.4) -- (2,1) -- (1,0) -- (0.6,0.5) -- (1,1);
\end{tikzpicture}
\\
\hline
\hline

3 & $c$ & $+-{}-+$ & 
% Case 3 knot before
\begin{tikzpicture}[scale = .5, rounded corners = 1mm]
\draw[white] (-.2,-.2) rectangle (3.2,2.2);
\draw (0,0) -- (1,1) -- (1.3,.7);
\draw (0,1) -- (0.3,0.7);
\draw (0.7,0.3) -- (1,0) -- (2,1) -- (2.3,0.7);
\draw (1.7,0.3) -- (2,0) -- (3,1) -- (2,2) -- (0,2);
\draw (2.7,0.3) -- (3,0);

\draw[->] (0.5, 0.5) -- (0.9, 0.9);
\draw[->] (0.7,0.3) -- (0.9,0.1);
\draw[->] (1.5, 0.5) -- (1.9,0.9);
\draw[->] (1.7,0.3) -- (1.9, 0.1);
\draw[->] (2.5,0.5) -- (2.9,0.9);
\draw[->] (2.7,0.3) -- (2.9, 0.1);
\end{tikzpicture}

 & % Case 3 Seifert state before
\begin{tikzpicture}[scale = .5, rounded corners = 1mm]
\draw[white] (-.2,-.2) rectangle (3.2,2.2);
\draw[->] (0,0) -- (.5,.4) -- (1,0) -- (1.5,.4) -- (2,0) -- (2.5,.4) -- (3,0);
\draw[->] (0,1) -- (.5,.6) -- (1,1) -- (1.5,.6) -- (2,1) -- (2.5,.6) -- (3,1) -- (2,2) -- (0,2);

\end{tikzpicture}
 \\
\hline
3 & $c-2$ & $+$ & 
 % Case 3 knot after
\begin{tikzpicture}[scale = .5, rounded corners = 1mm]
\draw[white] (-.2,-.2) rectangle (1.2,2.2);
\draw (0,1) -- (.3,.7);
\draw (.7,.3) -- (1,0);
\draw (0,0) -- (1,1) -- (0,2);

\draw[->] (.5,.5) -- (.9,.9);
\draw[->] (.7,.3) -- (.9,.1);

\end{tikzpicture}
 &  % Case 3 Seifert statte after
\begin{tikzpicture}[scale = .5, rounded corners = 1mm]
\draw[white] (-.2,-.2) rectangle (1.2,2.2);
\draw[->] (0,0) -- (.5,.4) -- (1,0);
\draw[->] (0,1) -- (.5,.6) -- (1,1) -- (0,2);

\end{tikzpicture}
 \\
\hline
\hline
4 & $c$ & $++-+$ & 
% Case 4 knot before
\begin{tikzpicture}[scale = .5, rounded corners = 1mm]
\draw[white] (-.2,-.2) rectangle (3.2,2.2);
\draw (0,0) -- (2,0) -- (3,1) -- (2,2) -- (1.7,1.7);
\draw (1.3,1.3) -- (1,1) -- (0,2);
\draw (0,1) -- (.3,1.3);
\draw (.7,1.7) -- (1,2) -- (2.3,.7);
\draw (2.7,0.3) -- (3,0);

\draw[->] (0.5, 1.5) -- (0.1, 1.9);
\draw[->] (0.7,1.7) -- (0.9,1.9);
\draw[->] (1.5, 1.5) -- (1.9,1.1);
\draw[->] (1.3,1.3) -- (1.1, 1.1);
\draw[->] (2.5,0.5) -- (2.9,0.9);
\draw[->] (2.7,0.3) -- (2.9, 0.1);
\end{tikzpicture}
& 
% Case 4 Seifert state before
\begin{tikzpicture}[scale = .5, rounded corners = 1mm]
\draw[white] (-.2,-.2) rectangle (3.2,2.2);
\draw[->] (0,0) -- (2,0) -- (2.5,.4) -- (3,0);
\draw[->] (0,1) -- (.4,1.5) -- (0,2);
\draw[->] (1,2) arc (90:-270:.4 cm and .5cm);
\draw[->] (2,1) -- (2.5,.6) -- (3,1) -- (2,2) -- (1.6,1.5) -- (2,1);
%\draw[->] (2,1) -- (1.6,1.5) -- (2,2) -- (3,1) -- (2.5,.6) -- (2,1);

\end{tikzpicture}
\\
\hline
4 & $c-2$ & $+$ & % Case 4 knot after
\begin{tikzpicture}[scale = .5, rounded corners = 1mm]
\draw[white] (-.2,-.2) rectangle (1.2,2.2);
\draw (0,1) -- (.3,.7);
\draw (.7,.3) -- (1,0);
\draw (0,0) -- (1,1) -- (0,2);

\draw[->] (.5,.5) -- (.9,.9);
\draw[->] (.7,.3) -- (.9,.1);

\end{tikzpicture}
 &  % Case 4 Seifert state after
\begin{tikzpicture}[scale = .5, rounded corners = 1mm]
\draw[white] (-.2,-.2) rectangle (1.2,2.2);
\draw[->] (0,0) -- (.5,.4) -- (1,0);
\draw[->] (0,1) -- (.5,.6) -- (1,1) -- (0,2);

\end{tikzpicture}
 \\
 \hline
\end{tabular}
\caption{Alternating diagrams and Seifert states corresponding to the cases in the proof of Theorem \ref{thm:Seifertrecursion}.}
\label{tab:Seifert}
\end{table}

\begin{theorem}
\label{thm:Seifertrecursion}
Let $s(c)$ be the total number of Seifert circles obtained when Seifert's algorithm is applied to the alternating $2$-bridge diagrams associated with words in $T(c)$. Then $s(c)$ satisfies the recursion $s(c)= s(c-1) + 2s(c-2) + 3t(c-2)$.
\end{theorem}

\begin{proof}
Following the ideas from earlier in this section, we consider the contributions to $s(c)$ from each of the four cases, calling these $s_1(c)$, $s_2(c)$, $s_3(c)$, and $s_4(c)$ so that $s(c)=s_1(c)+s_2(c)+s_3(c)+s_4(c)$. Refer to Table \ref{tab:Seifert} for pictures of each of the cases, where the orientations of the crossings are determined by Lemma \ref{lem:or1}.

In case 3, the final string $+-{}-+$ in a word with crossing number $c$ is replaced by $+$ in a new word with crossing number $c-2$. The partial Seifert states in the last column of Table \ref{tab:Seifert} before and after the replacement will have the same number of components when completed. Therefore $s_3(c) = s(c-2)$, the total number of Seifert circles from $T(c-2)$.

In case 4, the final string $++-+$ in a word with crossing number $c$ is replaced by $+$ in a new word with crossing number $c-2$. When the partial Seifert states in the last column of Table \ref{tab:Seifert} are completed, the state before the replacement will have two more components than the state after the replacement. Thus $s_4(c)=s(c-2)+2t(c-2)$, the total number of Seifert circles from $T(c-2)$ and additionally counting two circles for each element in $T(c-2)$.

In case 1,  the final string $+-+$ in a word with crossing number $c$ is replaced by a $++-$ in a new word with crossing number $c-1$. When the partial Seifert states in the last column of Table \ref{tab:Seifert} are completed, the state before the replacement will have one more component than the state after the replacement.
Thus $s_1(c)$ is equal to the sum of the total number of Seifert circles in words in $T(c-1)$ that end with $++-$ and $t_1(c)$, the number of words in case 1. The subset of $T(c-1)$ consisting of words ending with $++-$ can be partitioned into the subset of words ending in $-++-$ (case 3 for $c-1$ crossings) and the subset of words ending in $-{}-++-$ (case 2 for $c-1$ crossings). Thus the total number of Seifert circles is
\[s_1(c) = s_2(c-1) + s_3(c-1) + t_1(c) = s_2(c-1)+s_3(c-1)+2t(c-3).\]

In case 2, the final string $++ -{}-+$ in a word $w\in T(c)$ is replaced by $+-$, obtaining a diagram with $c-1$ crossings. The $(c-3)$rd run in $w$ is either a single $-$ or a double $-{}-$; we name these cases $2A$ and $2B$, respectively. So in case $2A$, the final string $-++-{}-+$ in $w$ is replaced with $-+-$, and in case $2B$, the final string $-{}-++-{}-+$ in $w$ is replaced with $-{}-+-$. Let $s_{2A}(c)$ and $s_{2B}(c)$ be the number of Seifert circles coming from words in $T(c)$ in cases $2A$ and $2B$, respectively.

In case $2A$, Table \ref{tab:Seifert} shows that the Seifert state before the replacement has one more component than the Seifert state after the replacement. Because the replacement words end with $-+-$, the set of replacement words for case $2A$ is case 1 for $c-1$ crossings. Therefore $s_{2A}(c) = s_1(c-1) + t_1(c-1)$.  In case $2B$, Table \ref{tab:Seifert} shows that the Seifert state before the replacement has one fewer component than the Seifert state after the replacement. Because the replacement words end with $-{}-+-$, the set of replacement words is case 4 for $c-1$ crossings. Thus $s_{2B}(c) = s_4(c-1) - t_4(c-1)$.

Lemma \ref{lem:countcases} implies that $t_1(c-1) = 2t(c-4)$ and $t_4(c-1)=t(c-3)$. Therefore,
\begin{align*}
s_2(c) = & \; s_{2A}(c) + s_{2B}(c)\\
= & \; [s_1(c-1) + t_1(c-1)] + [s_4(c-1) - t_4(c-1)]\\
= & \; s_1(c-1) + s_4(c-1) -t(c-3) + 2t(c-4) .
\end{align*}

Hence, we have 
\begin{align*}
s(c) = & \; s_1(c)+s_2(c)+s_3(c)+s_4(c)\\
= & \; [s_2(c-1) + s_3(c-1) + 2t(c-3)] + [s_1(c-1) + s_4(c-1) -t(c-3) + 2t(c-4)]\\ & \;+ s(c-2) + s(c-2)+ 2t(c-2)\\
 = &\; \sum_{i=1}^4 s_i(c-1) + 2s(c-2) + [t(c-3) + 2t(c-4)] + 2t(c-2)\\
 = & \;  s(c-1) + 2s(c-2) + 3t(c-2).
\end{align*}
\end{proof}

\subsection{Palindromic case}
\label{subsec:palindromic}
Recall that $T_p(c)$ is the set of strings in $\{+,-\}$ of palindromic type for crossing number $c$.  Alternatively we may abuse notation by using $T_p(c)$ to refer to the set of the corresponding alternating knot diagrams.  Let $t_p(c)$ be the number of elements in the set $T_p(c)$.  
Theorem \ref{thm:list} states that all 2-bridge knots are counted twice in $T(c)$ \emph{except} for words of palindromic type in $T_p(c)$, which are only counted once.  For odd $c$, such words are indeed palindromes; for even $c$, the words need to be read backwards and then have all $+$'s changed to $-$'s and vice versa. Equation \ref{eq:avseifert} states that the average number of Seifert circles in an alternating diagram of a $2$-bridge knot with crossing number $c$ is $\overline{s}_c = \frac{s(c) + s_p(c)}{4|\mathcal{K}_c|}$. In this subsection we mirror the previous subsection to obtain a recursive formula for $s_p(c)$.

In the discussion below, we consider separately the cases of odd $c$ and even $c$; so let us define $c=2i+1$ and $c=2i$ in these cases, respectively.  Let $T_{po}(i)$ and $T_{pe}(i)$ be the respective sets, and let $t_{po}(i)$ and $t_{pe}(i)$ be the number of elements in $T_{po}(i)$ and $T_{pe}(i)$, respectively.

\begin{proposition}
\label{prop:numberpalindromic}
The number $t_p(c)$ of words of palindromic type in $T_p(c)$ satisfies the recursion $t_p(c)=t_p(c-2)+2t_p(c-4)$. Moreover,
\[t_p(c) = \begin{cases} 
J\left(\frac{c-2}{2}\right) =  \frac{2^{(c-2)/2} - (-1)^{(c-2)/2}}{3} & \text{if $c$ is even and}\\
J\left(\frac{c-1}{2}\right) = \frac{2^{(c-1)/2} - (-1)^{(c-1)/2}}{3} & \text{if $c$ is odd,}\\
\end{cases}
\]
where $J(n)$ is the $n$th Jacobsthal number.
\end{proposition}

When restricting parity, this follows a similar pattern as the recursion $t(c)=t(c-1)+2t(c-2)$ for $t(c)$.

\begin{proof}
We proceed by induction on $c$. The base cases $t_p(3)=t_p(4)=1$ and $t_p(5)=t_p(6)=1$ are satisfied by the proof of Proposition \ref{prop:countterms} and Table \ref{tab:c456}, respectively.

Consider separately the number of terms $t_{pe}(i)$ and $t_{po}(i)$ for $c=2i$ and $c=2i+1$, respectively, with the goal of showing the recursion mentioned in the remark above.  Suppose that $c=2i$ is even, and let $w\in T_{pe}(i)$. Since $w=\overline{r}(w)$, the $i$th and $(i+1)$st runs must have the same length but be opposite symbols, and the $(i-1)$st and $(i+2)$nd runs must have the same length but be opposite symbols. Without loss of generality, assume $i$ is even; then the $(i-1)$st run is a single $+$ or double $+$, and the $i$th run is a single $-$ or a double $-{}-$. 
Then the $(i-1)$st and $i$th runs must be exactly one of the following cases:
\begin{itemize}
	\item[(1$_{pe}$)] a single $+$ followed by a single $-$,
	\item[(2$_{pe}$)] a double $++$ followed by a double $-{}-$,
	\item[(3$_{pe}$)] a single $+$ followed by a double $-{}-$, or
	\item[(4$_{pe}$)] a double $++$ followed by a single $-$.
\end{itemize}

If we replace the center four runs $+-+-$ in case 1$_{pe}$ with $++-{}-$, then two crossings can be removed without changing the length.  If we replace the center four runs $++-{}-++-{}-$ in case 2$_{pe}$ with $+-$, then two crossings can be removed without changing the length requirement modulo 3.  Furthermore, in both cases this does not affect the parity of the number of crossings, and we are left with $c-2$ crossings. These two cases partition $T_p(c-2)$, the subset of $T(c-2)$ consisting words of palindromic type with crossing number $c-2$.  In case 2$_{pe}$, the $i$th run is a single, and in case 1$_{pe}$, it is a double.  Thus these two cases together contribute $t_p(c-2)$ words.

The strings $-++-{}-+$ and $-{}-+-++$ in positions $i-1$ through $i+2$ in cases 3$_{pe}$ and 4$_{pe}$ each have length six, which is convenient for our model.  If these six crossings are removed, then the length requirement modulo 3 remains satisfied.  What is left after removal in each case is the set $T_p(c-4)$, and so cases 3 and 4 contribute $2t_p(c-4)$ words. Hence if $c$ is even, then $t_p(c)=t_p(c-2) + 2t_p(c-4)$.

Since $t_p(4)=t_p(6)=1$ and $t_p(c)=t_p(c-2) + 2t_p(c-4)$ when $c$ is even, the sequence $t_p(2n+2)$ for $n=1,2,\dots$ is the Jacobsthal sequence. Thus, if $c$ is even, then 
\[t_p(c) = J\left(\frac{c-2}{2}\right) =  \frac{2^{(c-2)/2} - (-1)^{(c-2)/2}}{3}.\]

Now suppose $c=2i+1$ is odd, and let $w\in T_{po}(i)$. Since $c=2i+1$ is odd, the $(i+1)$st run is in the middle of the word, and since $w=r(w)$, the $i$th run and the $(i+2)$nd run are the same length and consist of the same symbol. 
Without loss of generality, assume $i$ is odd; thus the $(i+1)$st run is a single $-$ or double $-{}-$.  Then the $i$th through $(i+2)$nd runs must be exactly one of the following cases:
\begin{itemize}
	\item[(1$_{po}$)] a single $+$ followed by a double $-{}-$ followed by a single $+$,
	\item[(2$_{po}$)] a double $++$ followed by a single $-$ followed by a double $++$,
	\item[(3$_{po}$)] a single $+$ followed by a single $-$ followed by a single $+$, or
	\item[(4$_{po}$)] a double $++$ followed by a double $-{}-$ followed by a double $++$.
\end{itemize}

If we replace the string $+--+$ in case 1$_{po}$ with a single $+$ or if we replace the string $++-++$ in case 2$_{po}$  with a double $++$, then two crossings can be removed without changing the length requirement modulo 3.  Furthermore this does not affect the parity of the number of crossings, and we are left with $c-2$ crossings. These two cases partition $T_p(c-2)$ the subset of words of palindromic type with crossing number $c-2$.  In case 1$_{po}$ the middle run is a single and in case 2$_{po}$ it is a double.  Thus these two cases together contribute $t_p(c-2)$ words.

In case $3_{po}$, the $i$th through $(i+2)$nd runs are $+-+$. There are two possibilities for the $(i-1)$st through the $(i+3)$rd runs: either $ - + - + -$ or $-{}- + - + -{}-$. The string $ - + - + -$ can be replaced with $-{}-$, and the string $-{}- + - + -{}-$ can be replaced with $-$. These replacements respect the length condition modulo 3 and result in words of palindromic type with crossing number $c-4$ in $T_p(c-4)$. In the first replacement, the middle run is a double $-{}-$, and in the second replacement, the middle run is a single $-$; therefore, these two subcases partition $T_p(c-4)$ and contribute $t_p(c-4)$ words.

In case $4_{po}$, the $i$th through $(i+2)$nd runs are $++-{}-++$. There are two possibilities for the $(i-1)$st through the $(i+3)$rd runs: either $-++-{}-++-$ or $-{}- ++ -{}- ++ -{}-$. The string $-++-{}-++-$ can be replaced with $-{}-$, and the string $-{}- ++ -{}- ++ -{}-$ can be replaced with $-$. These replacements respect the length condition modulo 3 and result in words of palindromic type with crossing number $c-4$ in $T_p(c-4)$. In the first replacement, the middle run is a double $-{}-$, and in the second replacement, the middle run is a single $-$; therefore, these two subcases partition $T_p(c-4)$ and contribute $t_p(c-4)$ words. Thus when $c$ is odd, $t_p(c) = t_p(c-2)+2t_p(c-4)$. 

Since $t_p(3)=t_p(5)=1$ and $t_p(c) = t_p(c-2)+2t_p(c-4)$ when $c$ is odd, the sequence $t_p(2n+1)$ for $n=1,2,\dots$ is the Jacobsthal sequence. Thus, if $c$ is odd, then 
\[t_p(c) = J\left(\frac{c-1}{2}\right) = \frac{2^{(c-1)/2} - (-1)^{(c-1)/2}}{3}.\]
\end{proof}

\begin{example}
\label{ex:c9counttermsp}
Table \ref{tab:c579p} shows the words of palindromic type in $T_p(5)$, $T_p(7)$, and $T_p(9)$. Note that for $c=9$, we have even $i$, which is opposite the discussion in the proof above.  Subwords of words in $T_p(9)$ in parentheses are replaced according to the proof of Proposition \ref{prop:numberpalindromic} to obtain the words on the left in either $T_p(5)$ or $T_p(7)$. We see that $t_p(9) = t_p(7) + 2t_p(5)$.
\end{example}

\begin{center}
\begin{table}[h]
\begin{tabular}{|c|c||c|c|}
\hline
%$c=3$ & $+-()-+$ 				& $+-(+-+)-+$ 				& \\
%\cline{1-2}
%$c=3$ & $+-()-+$ 				& $+-(++-{}-++)-+$ 			& $c=7$\\
%\cline{1-2}
%$c=5$	& $+-{}-(+)-{}-+$ & $+-{}-(+-{}-+)-{}-+$ 	& \\
%
$T_p(5)$ & $+-{}-(+)-{}-+$ 				& $+-{}-(++-{}-++-{}-++)-{}-+$ 				& \\
\cline{1-2}
$T_p(5)$ & $+-{}-(+)-{}-+$ 				& $+-{}-(++-+-++)-{}-+$ 			& \\
\cline{1-2}
\multirow{3}{*}{$T_p(7)$}			& $+-+(-)+-+$ 					& $+-+(-++-)+-+$ 	& $T_p(9)$\\
														& $+-++(-{}-)++-+$ 			& $+-++(-{}-+-{}-)++-+$ 				& \\
														& $+-{}-+(-{}-)+-{}-+$ 	& $+-{}-+(-{}-+-{}-)+-{}-+$ 		& \\
\hline
\end{tabular}
\caption{The sets $T_p(5)$, $T_p(7)$ and $T_p(9)$ with the subwords in parentheses replaced as in the proof of Proposition \ref{prop:numberpalindromic}.}
\label{tab:c579p}
\end{table}
\end{center}

\begin{example}
\label{ex:c10counttermsp}
Table \ref{tab:c6810p} shows the words of palindromic type in $T_p(6)$, $T_p(8)$, and $T_p(10)$.  Note that for $c=10$, we have odd $i$, which is opposite the discussion in the proof above.  Subwords of words in $T_p(10)$ in parentheses are replaced according to the proof of Proposition \ref{prop:numberpalindromic} to obtain the words on the left in either $T_p(6)$ or $T_p(8)$. We see that $t_p(10) = t_p(8) + 2t_p(6)$.
\end{example}

\begin{center}
\begin{table}[h]
\begin{tabular}{|c|c||c|c|}
\hline
%$c=4$ & $+-()+-$ 				& $+-(+-+)+-$ 				& \\
%\cline{1-2}
%$c=4$ & $+-()+-$ 				& $+-+(-{}-+)+-$ 			& $c=8$ \\
%\cline{1-2}
%$c=6$	& $+-{}-(+)-{}-+$ & $+-{}-(+-{}-+)-{}-+$ 	& \\
%
$T_p(6)$ & $+-{}-++()-{}-++-$ 				& $+-{}-++(-++-{}-+)-{}-++-$ 				& \\
\cline{1-2}
$T_p(6)$ & $+-{}-++()-{}-++-$ 				& $+-{}-++(--+-++)-{}-++-$ 					& \\
\cline{1-2}
\multirow{3}{*}{$T_p(8)$}			& $+-+(--++)-+-$ 		& $+-+(-+-+)-+-$ 						& $T_p(10)$\\
														& $+-++(-+)-{}-+-$ 	& $+-++(--++-{}-++)-{}-+-$ 	& \\
														& $+-{}-+(-+)-++-$ 	& $+-{}-+(--++-{}-++)-++-$ 		& \\
\hline
\end{tabular}
\caption{The sets $T_p(6)$, $T_p(8)$, and $T_p(10)$ with the subwords in parentheses replaced as in the proof of Proposition \ref{prop:numberpalindromic}.}
\label{tab:c6810p}
\end{table}
\end{center}

We are now ready to prove the recursive formula for $s_p(c)$, the total number of Seifert circles from $T_p(c)$.
\begin{theorem}
\label{thm:Seifertrecursionpalindrome}
Let $s_p(c)$ be the total number of Seifert circles over all 2-bridge knots of palindromic type with crossing number $c$ for all knots appearing in $T_p(c)$.  Then $s_p(c)$ satisfies the recursion $s_p(c)= s_p(c-2) + 2s_p(c-4) + 6t_p(c-4)$.
\end{theorem}

\begin{proof}
As in the proof of Proposition \ref{prop:numberpalindromic}, we consider separately the cases for even $c=2i$ and odd $c=2i+1$ crossing number, with notation $s_{pe}(i)=s_p(2i)$ and $s_{po}(i)=s_p(2i+1)$. Suppose $c=2i$ is even. In the same spirit as Lemma \ref{lem:countcases}, define $t_{pe1}(i)$, $t_{pe2}(i)$, $t_{pe3}(i)$, and $t_{pe4}(c)$ to be the number of words in cases $1_{pe}$, $2_{pe}$, $3_{pe}$, and $4_{pe}$, respectively.  Similarly, as in the proof of Theorem \ref{thm:Seifertrecursion}, define $s_{pe1}(i)$, $s_{pe2}(i)$, $s_{pe3}(i)$, and $s_{pe4}(c)$ to be the number of Seifert circles coming from words in cases $1_{pe}$, $2_{pe}$, $3_{pe}$, and $4_{pe}$, respectively. Then $s_{pe}(i)=s_{pe1}(i)+s_{pe2}(i)+s_{pe3}(i)+s_{pe4}(i)$.  Refer to Table \ref{tab:SeifertPalindromeEven} for pictures of each of the cases, where the orientations of the crossings are determined by Lemma \ref{lem:or1}.

In case 1$_{pe}$, the center string $+-+-$ in a word with crossing number $c$ is replaced by $++-{}-$ in a new word with crossing number $c-2$, and in case $2_{pe}$, the center string $++-{}-++-{}-$ in a word with crossing number $c$ is replaced by $+-$ in a new word with crossing number $c-2$. Lemma \ref{lem:or1} and the first four rows in Table \ref{tab:SeifertPalindromeEven} imply that the only changes caused by these replacements are the removal of two horizontally-oriented crossings. The Seifert states before and after the replacements have the same number of components. Since the center strings $+-$ and $++-{}-$ partition $T_{pe}(i-1)$, it follows that $s_{pe1}(i)+s_{pe2}(i)=s_{pe}(i-1)$.

As in the odd palindromic case of the proof of Proposition \ref{prop:numberpalindromic} above, we split cases 3$_{pe}$ and 4$_{pe}$ into two subcases called $A$ and $B$ depending on whether the ($i-2$)nd run is a single $-$ or a double $-{}-$, respectively. 

In case 3A$_{pe}$, the center string $-+-{}-++-+$ in a word with crossing number $c$ is replaced by $-+$ in a new word with crossing number $c-4$.  Lemma \ref{lem:or1} and the fifth and sixth rows in Table \ref{tab:SeifertPalindromeEven} imply that the Seifert state after the replacement has four fewer components than the Seifert state before the replacement. So in order to count $s_{pe3A}(i)$ we need to count the number of words in this case. The center string in the new word with crossing number $c-4$ is $-+$. The cases that have such a center word are 1$_{pe}$ and 3$_{pe}$ for crossing number $c-4$.  Thus $s_{pe3A}(i)=(s_{pe1}(i-2)+s_{pe3}(i-2))+4(t_{pe1}(i-2)+t_{pe3}(i-2))$.

In case 3B$_{pe}$, the center string $-{}-+-{}-++-++$ in a word with crossing number $c$ is replaced by $-{}-++$ in a new word with crossing number $c-4$.  Lemma \ref{lem:or1} and the seventh and eighth rows in Table \ref{tab:SeifertPalindromeEven} imply that the Seifert state after the replacement has two fewer components than the Seifert state before the replacement. So in order to count $s_{pe3B}(i)$ we need to count the number of words in this case.  The center string in the new word with crossing number $c-4$ is $-{}-++$. The cases that have such a center word are 2$_{pe}$ and 4$_{pe}$ for crossing number $c-4$.  Thus $s_{pe3B}(i)=(s_{pe2}(i-2)+s_{pe4}(i-2))+2(t_{pe2}(i-2)+t_{pe4}(i-2))$.

In case 4A$_{pe}$, the center string $-++-+-{}-+$ in a word with crossing number $c$ is replaced by $-+$ in a new word with crossing number $c-4$. Lemma \ref{lem:or1} and the ninth and tenth rows in Table \ref{tab:SeifertPalindromeEven} imply that the Seifert state after the replacement has two fewer components than the Seifert state before the replacement.  By a similar argument as case 3A$_{pe}$, we get $s_{pe4A}(i)=(s_{pe1}(i-2)+s_{pe3}(i-2))+2(t_{pe1}(i-2)+t_{pe3}(i-2))$.

In case 4B$_{pe}$, the center string $-{}-++-+-{}-++$ in a word with crossing number $c$ is replaced by $-{}-++$ in a new word with crossing number $c-4$. Lemma \ref{lem:or1} and the last two rows in Table \ref{tab:SeifertPalindromeEven} imply that the Seifert state after the replacement has four fewer components than the Seifert state before the replacement. By a similar argument as case 3B$_{pe}$, we get $s_{pe4B}(i)=(s_{pe2}(i-2)+s_{pe4}(i-2))+4(t_{pe2}(i-2)+t_{pe4}(i-2))$.

Thus
\begin{align*}
s_{pe3}(i) + s_{pe4}(i) = & \; s_{pe3A}(i) + s_{pe4B}(i) + s_{pe3B}(i) + s_{pe4A}(i) \\
= & \; (s_{pe1}(i-2)+s_{pe3}(i-2))+4(t_{pe1}(i-2)+t_{pe3}(i-2)) \\
& \; + (s_{pe2}(i-2)+s_{pe4}(i-2))+4(t_{pe2}(i-2)+t_{pe4}(i-2))\\
 & \; + (s_{pe2}(i-2)+s_{pe4}(i-2))+2(t_{pe2}(i-2)+t_{pe4}(i-2))\\
 & \; + (s_{pe1}(i-2)+s_{pe3}(i-2))+2(t_{pe1}(i-2)+t_{pe3}(i-2))\\
 = & \; 2\sum_{j=1}^4 s_{pej}(i-2) + 6 \sum_{j=1}^4 t_{pej}(i-2)\\
 = & \; 2s_{pe}(i-2) + 6 t_{pe}(i-2).
\end{align*}
Concluding the even length case, we have
\[s_{pe}(i) = \sum_{j=1}^4 s_{pej}(i) = s_{pe}(i-1) + 2s_{pe}(i-2) + 6 t_{pe}(i-2).\]

When $c=2i+1$ is odd, one can prove that $s_{po}(i) =  s_{po}(i-1) + 2s_{po}(i-2) + 6 t_{po}(i-2)$ in a similar fashion. The interested reader can work out the details from Table \ref{tab:SeifertPalindromeOdd}. Since $s_{pe}(i)=s_p(2i)$ and $s_{po}(i)=s_p(2i+1)$, it follows that
\[s_p(c) = s_p(c-2) + 2s_p(c-4)+6t_p(c-4).\]

\end{proof}

\begin{table}
\begin{tabular}{|c|c||c|c|c|}
\hline
Case & Crossing & String & Alternating Diagram & Seifert state \\
 & Number & & & \\
\hline
\hline
1$_{pe}$ & $c$ & \tiny{$+-+-$} & 
% Case 1 knot before
\begin{tikzpicture}[scale=.4]
\draw[white] (-.2,-.7) rectangle (10.2,2.7);
\draw (1,-.5) rectangle (3,2.5);
\draw (2,1) node{$R$};
\draw (8,1) node[rotate = 180]{$\overline{R}$};
\draw (7,-.5) rectangle (9,2.5);
\draw (0,0) -- (1,0);
\draw (9,2) -- (10,2);
\begin{scope}[rounded corners = 1mm]
	\draw (3,0) -- (4.3,1.3);
	\draw (3,1) -- (3.3,.7);
	\draw (3,2) -- (4,2) -- (5.3,.7);
	\draw (4.7,1.7) -- (5,2) -- (6,2) -- (7,1);
	\draw (3.7,.3) -- (4,0) -- (5,0) -- (6.3,1.3);
	\draw (6.7,1.7) -- (7,2);
	\draw (5.7,.3) -- (6,0) -- (7,0);
\end{scope}
\draw[->] (3.5,.5) -- (3.9,.9);
\draw[->] (3.7,.3) -- (3.9,.1);
\draw[->] (4.5,1.5) -- (4.1,1.9);
\draw[->] (4.7,1.7) -- (4.9,1.9);
\draw[->] (5.5,.5) -- (5.9,.9);
\draw[->] (5.3,.7) -- (5.1,.9);
\draw[->] (6.5,1.5) -- (6.9,1.1);
\draw[->] (6.7,1.7) -- (6.9,1.9);

\end{tikzpicture}
 &
 % Case 1 Seifert state before
 \begin{tikzpicture}[scale=.4]
\draw[white] (-.2,-.7) rectangle (10.2,2.7);
\draw (1,-.5) rectangle (3,2.5);
\draw (7,-.5) rectangle (9,2.5);
\draw (0,0) -- (1,0);
\draw (9,2) -- (10,2);
\begin{scope}[rounded corners = 1mm]
	\draw[->] (3,0) -- (3.5,.4) -- (4,0) -- (5,0) -- (5.4,.5) -- (5,1) -- (4.6,1.5) -- (5,2) -- (6,2) -- (6.5,1.6) -- (7,2);
	\draw[->] (3,2) -- (4,2) -- (4.4,1.5) -- (4,1) -- (3.5,.6) -- (3,1);
	\draw[->] (7,0) -- (6,0) -- (5.6,.5) -- (6,1) -- (6.5,1.4) -- (7,1);
\end{scope}
\draw[densely dashed] (1,0) -- (3,0);
\draw[densely dashed, rounded corners=1mm] (3,2) -- (2.6,1.5) -- (3,1);
\draw[densely dashed] (7,2) -- (9,2);
\draw[densely dashed, rounded corners=1mm] (7,1) -- (7.4,.5) -- (7,0);

\end{tikzpicture}
\\
\hline
1$_{pe}$ & $c-2$ & \tiny{$++ -{}-$} &
% Case 1 knot after
\begin{tikzpicture}[scale=.4]
\draw[white] (-.2,-.7) rectangle (8.2,2.7);
\draw (1,-.5) rectangle (3,2.5);
\draw (2,1) node{$R$};
\draw (6,1) node[rotate = 180]{$\overline{R}$};
\draw (5,-.5) rectangle (7,2.5);
\draw (0,0) -- (1,0);
\draw (7,2) -- (8,2);
\begin{scope}[rounded corners = 1mm]
	\draw (3,0) -- (4,0) -- (5,1);
	\draw (3,1) -- (3.3,1.3);
	\draw (3.7,1.7) -- (4,2) -- (5,2);
	\draw (3,2) -- (4.3,.7);
	\draw (4.7,.3) -- (5,0);
\end{scope}

\draw[->] (3.5,1.5) -- (3.1,1.9);
\draw[->] (3.7,1.7) -- (3.9,1.9);
\draw[->] (4.5,.5) -- (4.9,.9);
\draw[->] (4.3,.7) -- (4.1,.9);

\end{tikzpicture}
		&
% Case 1 Seifert state after
 \begin{tikzpicture}[scale=.4]
\draw[white] (-.2,-.7) rectangle (8.2,2.7);
\draw (1,-.5) rectangle (3,2.5);
\draw (5,-.5) rectangle (7,2.5);
\draw (0,0) -- (1,0);
\draw (7,2) -- (8,2);
\begin{scope}[rounded corners = 1mm]
	\draw[->] (3,1) -- (3.4,1.5) -- (3,2);
	\draw[->] (5,0) -- (4.6,.5) -- (5,1);
	\draw[->] (3,0) -- (4,0) -- (4.4,.5) -- (4,1) -- (3.6,1.5) -- (4,2) -- (5,2);
\end{scope}
\draw[densely dashed] (1,0) -- (3,0);
\draw[densely dashed, rounded corners=1mm] (3,2) -- (2.6,1.5) -- (3,1);
\draw[densely dashed] (5,2) -- (7,2);
\draw[densely dashed, rounded corners=1mm] (5,1) -- (5.4,.5) -- (5,0);

\end{tikzpicture}
		
 \\
\hline
\hline
2$_{pe}$ & $c$ & \tiny{$++-{}-++-{}-$} &
% Case 2 knot before
\begin{tikzpicture}[scale=.4]
\draw[white] (-.2,-.7) rectangle (10.2,2.7);
\draw (1,-.5) rectangle (3,2.5);
\draw (2,1) node{$R$};
\draw (8,1) node[rotate = 180]{$\overline{R}$};
\draw (7,-.5) rectangle (9,2.5);
\draw (0,0) -- (1,0);
\draw (9,2) -- (10,2);
\begin{scope}[rounded corners = 1mm]
	\draw (3,0) -- (4,0) -- (5.3,1.3);
	\draw (5.7,1.7) -- (6,2) --(7,2);
	\draw (3,1) -- (3.3,1.3);
	\draw (3.7,1.7) -- (4,2) -- (5,2) -- (6.3,.7);
	\draw (6.7,.3) -- (7,0);
	\draw (3,2) -- (4.3,.7);
	\draw (4.7,.3) -- (5,0) -- (6,0) -- (7,1);
\end{scope}
\draw[->] (3.5,1.5) -- (3.9,1.1);
\draw[->] (3.7,1.7) -- (3.9,1.9);
\draw[->] (4.5,.5) -- (4.1,.1);
\draw[->] (4.7,.3) -- (4.9,.1);
\draw[->] (5.5,1.5) -- (5.9,1.1);
\draw[->] (5.3,1.3) -- (5.1,1.1);
\draw[->] (6.5,.5) -- (6.9,.9);
\draw[->] (6.7,.3) -- (6.9,.1);

\end{tikzpicture}

  & 
   % Case 2 Seifert state before
 \begin{tikzpicture}[scale=.4]
\draw[white] (-.2,-.7) rectangle (10.2,2.7);
\draw (1,-.5) rectangle (3,2.5);
\draw (7,-.5) rectangle (9,2.5);
\draw (0,0) -- (1,0);
\draw (9,2) -- (10,2);
\begin{scope}[rounded corners = 1mm]
	\draw[->] (3,2) -- (3.5,1.6) -- (4,2) -- (5,2) -- (5.4,1.5) -- (5,1) -- (4.6,.5) -- (5,0) -- (6,0) -- (6.5,.4) -- (7,0);
	\draw[->] (3,1) -- (3.5,1.4) -- (4,1) -- (4.4,.5) -- (4,0) -- (3,0);
	\draw[->] (7,2) -- (6,2) -- (5.6,1.5) -- (6,1) -- (6.5,.6) -- (7,1);
\end{scope}
\draw[densely dashed] (1,0) -- (3,2);
\draw[densely dashed, rounded corners=1mm] (3,1) -- (2.6,.5) -- (3,0);
\draw[densely dashed] (7,0) -- (9,2);
\draw[densely dashed, rounded corners=1mm] (7,2) -- (7.4,1.5) -- (7,1);

\end{tikzpicture}
  \\
\hline
2$_{pe}$ & $c-2$ & \tiny{$+-$} & 
% Case 2 knot after
\begin{tikzpicture}[scale=.4]
\draw[white] (-.2,-.7) rectangle (8.2,2.7);
\draw (1,-.5) rectangle (3,2.5);
\draw (2,1) node{$R$};
\draw (6,1) node[rotate = 180]{$\overline{R}$};
\draw (5,-.5) rectangle (7,2.5);
\draw (0,0) -- (1,0);
\draw (7,2) -- (8,2);
\begin{scope}[rounded corners = 1mm]
	\draw (3,0) -- (4.3,1.3);
	\draw (4.7,1.7) -- (5,2);
	\draw (3,1) -- (3.3,.7);
	\draw (3,2) -- (4,2) -- (5,1);
	\draw (3.7,.3) -- (4,0) -- (5,0);
\end{scope}

\draw[->] (3.5,.5) -- (3.1,.1);
\draw[->] (3.7,.3) -- (3.9,.1);
\draw[->] (4.5,1.5) -- (4.9,1.1);
\draw[->] (4.3,1.3) -- (4.1,1.1);

\end{tikzpicture}
		&
% Case 2 Seifert state after
 \begin{tikzpicture}[scale=.4]
\draw[white] (-.2,-.7) rectangle (8.2,2.7);
\draw (1,-.5) rectangle (3,2.5);
\draw (5,-.5) rectangle (7,2.5);
\draw (0,0) -- (1,0);
\draw (7,2) -- (8,2);
\begin{scope}[rounded corners = 1mm]
	\draw[->] (3,1) -- (3.4,.5) -- (3,0);
	\draw[->] (5,2) -- (4.6,1.5) -- (5,1);
	\draw[->] (3,2) -- (4,2) -- (4.4,1.5) -- (4,1) -- (3.6,.5) -- (4,0) -- (5,0);
\end{scope}
\draw[densely dashed] (1,0) -- (3,2);
\draw[densely dashed, rounded corners=1mm] (3,0) -- (2.6,.5) -- (3,1);
\draw[densely dashed] (5,0) -- (7,2);
\draw[densely dashed, rounded corners=1mm] (5,1) -- (5.4,1.5) -- (5,2);

\end{tikzpicture}
		 \\
\hline
\hline
3A$_{pe}$ & $c$ & \tiny{$-+--++-+$} &
% Case 3A knot before
\begin{tikzpicture}[scale=.4]
\draw[white] (-.2,-.7) rectangle (12.2,2.7);
\draw (1,-.5) rectangle (3,2.5);
\draw (2,1) node{$R$};
\draw (10,1) node[rotate = 180]{$\overline{R}$};
\draw (9,-.5) rectangle (11,2.5);
\draw (0,0) -- (1,0);
\draw (11,2) -- (12,2);
\begin{scope}[rounded corners = 1mm]
	\draw (3,0) -- (4,0) -- (5,1) -- (5.3,.7);
	\draw (5.7,.3) -- (6,0) -- (8,0) -- (9,1);
	\draw (3,1) -- (3.3,1.3);
	\draw (3.7,1.7) -- (4,2) -- (6,2) -- (7,1) -- (7.3,1.3);
	\draw (7.7,1.7) -- (8,2) -- (9,2);
	\draw (3,2) -- (4.3,.7);
	\draw (4.7,.3) -- (5,0) -- (6.3,1.3);
	\draw (6.7,1.7) -- (7,2) -- (8.3,.7);
	\draw (8.7,.3) -- (9,0);
\end{scope}
\draw[->] (3.5,1.5) -- (3.1,1.9);
\draw[->] (3.7,1.7) -- (3.9,1.9);
\draw[->] (4.5,.5) -- (4.9,.9);
\draw[->] (4.3,.7) -- (4.1,.9);
\draw[->] (5.5,.5) -- (5.1,.1);
\draw[->] (5.7,.3) -- (5.9,.1);
\draw[->] (6.5,1.5) --(6.9,1.1);
\draw[->] (6.3,1.3) -- (6.1,1.1);
\draw[->] (7.5,1.5) -- (7.1,1.9);
\draw[->] (7.7,1.7) -- (7.9,1.9);
\draw[->] (8.5,.5) -- (8.9,.9);
\draw[->] (8.3,.7) -- (8.1,.9);

\end{tikzpicture}
& 
% Case 3A Seifert stae before
\begin{tikzpicture}[scale=.4]
\draw[white] (-.2,-.7) rectangle (12.2,2.7);
\draw (1,-.5) rectangle (3,2.5);
\draw (9,-.5) rectangle (11,2.5);
\draw (0,0) -- (1,0);
\draw (11,2) -- (12,2);
\begin{scope}[rounded corners = 1mm]
\draw[->] (3,0) -- (4,0) -- (4.4,.5) -- (4,1) -- (3.6,1.5) -- (4,2) -- (6,2) -- (6.4,1.5) -- (6,1) --(5.6,.5) -- (6,0) -- (8,0) -- (8.4,.5) -- (8,1) -- (7.6,1.5) -- (8,2) -- (9,2);
\draw[->] (3,1) -- (3.4,1.5) -- (3,2);
\draw[->] (9,0) -- (8.6,.5) -- (9,1);
\draw[->] (5,1) arc (90:-270:.4cm and .5cm);
\draw[->] (7,2) arc (90:450:.4cm and .5cm);
\end{scope}
\draw[densely dashed] (1,0) -- (3,0);
\draw[densely dashed, rounded corners =1mm] (3,1) -- (2.6,1.5) -- (3,2);
\draw[densely dashed] (9,2) -- (11,2);
\draw[densely dashed, rounded corners =1mm] (9,1) -- (9.4,.5) -- (9,0);

\end{tikzpicture}

\\
\hline
3A$_{pe}$ & $c-4$ & \tiny{$-+$} & 
% Case 3A knot after
\begin{tikzpicture}[scale=.4]
\draw[white] (-.2,-.7) rectangle (8.2,2.7);
\draw (1,-.5) rectangle (3,2.5);
\draw (2,1) node{$R$};
\draw (6,1) node[rotate = 180]{$\overline{R}$};
\draw (5,-.5) rectangle (7,2.5);
\draw (0,0) -- (1,0);
\draw (7,2) -- (8,2);
\begin{scope}[rounded corners = 1mm]
	\draw (3,0) -- (4,0) -- (5,1);
	\draw (3,1) -- (3.3,1.3);
	\draw (3.7,1.7) -- (4,2) -- (5,2);
	\draw (3,2) -- (4.3,.7);
	\draw (4.7,.3) -- (5,0);
\end{scope}

\draw[->] (3.5,1.5) -- (3.1,1.9);
\draw[->] (3.7,1.7) -- (3.9,1.9);
\draw[->] (4.5,.5) -- (4.9,.9);
\draw[->] (4.3,.7) -- (4.1,.9);

\end{tikzpicture}
		&
% Case 3A Seifert state after
 \begin{tikzpicture}[scale=.4]
\draw[white] (-.2,-.7) rectangle (8.2,2.7);
\draw (1,-.5) rectangle (3,2.5);
\draw (5,-.5) rectangle (7,2.5);
\draw (0,0) -- (1,0);
\draw (7,2) -- (8,2);
\begin{scope}[rounded corners = 1mm]
	\draw[->] (3,1) -- (3.4,1.5) -- (3,2);
	\draw[->] (5,0) -- (4.6,.5) -- (5,1);
	\draw[->] (3,0) -- (4,0) -- (4.4,.5) -- (4,1) -- (3.6,1.5) -- (4,2) -- (5,2);
\end{scope}
\draw[densely dashed] (1,0) -- (3,0);
\draw[densely dashed, rounded corners=1mm] (3,2) -- (2.6,1.5) -- (3,1);
\draw[densely dashed] (5,2) -- (7,2);
\draw[densely dashed, rounded corners=1mm] (5,1) -- (5.4,.5) -- (5,0);

\end{tikzpicture}
 \\
\hline
\hline
3B$_{pe}$ & $c$ & \tiny{$--+--++-++$} & 
% Case 3B knot before
\begin{tikzpicture}[scale=.4]
\draw[white] (-.2,-.7) rectangle (12.2,2.7);
\draw (1,-.5) rectangle (3,2.5);
\draw (2,1) node{$R$};
\draw (10,1) node[rotate = 180]{$\overline{R}$};
\draw (9,-.5) rectangle (11,2.5);
\draw (0,0) -- (1,0);
\draw (11,2) -- (12,2);
\begin{scope}[rounded corners = 1mm]
	\draw (3,0) -- (4,1) -- (4.3,.7);
	\draw (3.7,.3) -- (4,0) -- (5,1) -- (5.3,.7);
	\draw (5.7,.3) -- (6,0) -- (9,0);
	\draw (3,1) -- (3.3,.7);
	\draw (7.7,1.7) -- (8,2) -- (9,1);
	\draw (3,2) -- (6,2) -- (7,1) -- (7.3,1.3);
	\draw (4.7,.3) -- (5,0) -- (6.3,1.3);
	\draw (6.7,1.7) -- (7,2) -- (8,1) -- (8.3,1.3);
	\draw (8.7,1.7) -- (9,2);
\end{scope}
\draw[->] (3.5,.5) -- (3.1,.1);
\draw[->] (3.7,.3) -- (3.9,.1);
\draw[->] (4.5,.5) -- (4.9,.9);
\draw[->] (4.3,.7) -- (4.1,.9);
\draw[->] (5.5,.5) -- (5.1,.1);
\draw[->] (5.7,.3) -- (5.9,.1);
\draw[->] (6.5,1.5) --(6.9,1.1);
\draw[->] (6.3,1.3) -- (6.1,1.1);
\draw[->] (7.5,1.5) -- (7.1,1.9);
\draw[->] (7.7,1.7) -- (7.9,1.9);
\draw[->] (8.5,1.5) -- (8.9,1.1);
\draw[->] (8.3,1.3) -- (8.1,1.1);

\end{tikzpicture}
& 
% Case 3B Seifert stae before
\begin{tikzpicture}[scale=.4]
\draw[white] (-.2,-.7) rectangle (12.2,2.7);
\draw (1,-.5) rectangle (3,2.5);
\draw (9,-.5) rectangle (11,2.5);
\draw (0,0) -- (1,0);
\draw (11,2) -- (12,2);
\begin{scope}[rounded corners = 1mm]
\draw[->] (3,2) -- (6,2) -- (6.4,1.5) -- (6,1) -- (5.6,.5) -- (6,0) -- (9,0);
\draw[->] (3,1) -- (3.4,.5) -- (3,0);
\draw[->] (9,2) -- (8.6,1.5) -- (9,1);
\draw[->] (5,1) arc (90:-270:.4cm and .5cm);
\draw[->] (4,1) arc (90:450:.4cm and .5cm);
\draw[->] (7,2) arc (90:450:.4cm and .5cm);
\draw[->] (8,2) arc (90:-270:.4cm and .5cm);
\end{scope}
\draw[densely dashed] (1,0) -- (3,2);
\draw[densely dashed, rounded corners =1mm] (3,0) -- (2.6,.5) -- (3,1);
\draw[densely dashed] (9,0) -- (11,2);
\draw[densely dashed, rounded corners =1mm] (9,1) -- (9.4,1.5) -- (9,2);

\end{tikzpicture}

\\
\hline
3B$_{pe}$ & $c-4$ & \tiny{$--++$} &
% Case 3B knot after
\begin{tikzpicture}[scale=.4]
\draw[white] (-.2,-.7) rectangle (8.2,2.7);
\draw (1,-.5) rectangle (3,2.5);
\draw (2,1) node{$R$};
\draw (6,1) node[rotate = 180]{$\overline{R}$};
\draw (5,-.5) rectangle (7,2.5);
\draw (0,0) -- (1,0);
\draw (7,2) -- (8,2);
\begin{scope}[rounded corners = 1mm]
	\draw (3,0) -- (4.3,1.3);
	\draw (4.7,1.7) -- (5,2);
	\draw (3,1) -- (3.3,.7);
	\draw (3,2) -- (4,2) -- (5,1);
	\draw (3.7,.3) -- (4,0) -- (5,0);
\end{scope}

\draw[->] (3.5,.5) -- (3.1,.1);
\draw[->] (3.7,.3) -- (3.9,.1);
\draw[->] (4.5,1.5) -- (4.9,1.1);
\draw[->] (4.3,1.3) -- (4.1,1.1);

\end{tikzpicture}
		&
% Case 3B Seifert state after
 \begin{tikzpicture}[scale=.4]
\draw[white] (-.2,-.7) rectangle (8.2,2.7);
\draw (1,-.5) rectangle (3,2.5);
\draw (5,-.5) rectangle (7,2.5);
\draw (0,0) -- (1,0);
\draw (7,2) -- (8,2);
\begin{scope}[rounded corners = 1mm]
	\draw[->] (3,1) -- (3.4,.5) -- (3,0);
	\draw[->] (5,2) -- (4.6,1.5) -- (5,1);
	\draw[->] (3,2) -- (4,2) -- (4.4,1.5) -- (4,1) -- (3.6,.5) -- (4,0) -- (5,0);
\end{scope}
\draw[densely dashed] (1,0) -- (3,2);
\draw[densely dashed, rounded corners=1mm] (3,0) -- (2.6,.5) -- (3,1);
\draw[densely dashed] (5,0) -- (7,2);
\draw[densely dashed, rounded corners=1mm] (5,1) -- (5.4,1.5) -- (5,2);

\end{tikzpicture}
 \\
\hline
\hline
4A$_{pe}$ & $c$ & \tiny{$-++-+--+$} & 
% Case 4A knot before
\begin{tikzpicture}[scale=.4]
\draw[white] (-.2,-.7) rectangle (12.2,2.7);
\draw (1,-.5) rectangle (3,2.5);
\draw (2,1) node{$R$};
\draw (10,1) node[rotate = 180]{$\overline{R}$};
\draw (9,-.5) rectangle (11,2.5);
\draw (0,0) -- (1,0);
\draw (11,2) -- (12,2);
\begin{scope}[rounded corners = 1mm]
	\draw (3,0) -- (6,0) -- (7,1) -- (7.3,.7);
	\draw (7.7,.3) -- (8,0) -- (9,1);
	\draw (3,1) -- (3.3,1.3);
	\draw (3.7,1.7) -- (4,2) -- (5,1) -- (5.3,1.3);
	\draw (5.7,1.7) -- (6,2) -- (9,2);
	\draw (3,2) -- (4,1) -- (4.3,1.3);
	\draw (4.7,1.7) -- (5,2) -- (6.3,.7);
	\draw (6.7,.3) -- (7,0) -- (8,1) -- (8.3,.7);
	\draw (8.7,.3) -- (9,0);
\end{scope}
\draw[->] (3.5,1.5) -- (3.1,1.9);
\draw[->] (3.7,1.7) -- (3.9,1.9);
\draw[->] (4.5,1.5) -- (4.9,1.1);
\draw[->] (4.3,1.3) -- (4.1,1.1);
\draw[->] (5.5,1.5) -- (5.1,1.9);
\draw[->] (5.7,1.7) -- (5.9,1.9);
\draw[->] (6.5,.5) --(6.9,.9);
\draw[->] (6.3,.7) -- (6.1,.9);
\draw[->] (7.5,.5) -- (7.1,.1);
\draw[->] (7.7,.3) -- (7.9,.1);
\draw[->] (8.5,.5) -- (8.9,.9);
\draw[->] (8.3,.7) -- (8.1,.9);

\end{tikzpicture}
& 
% Case 4A Seifert state before
\begin{tikzpicture}[scale=.4]
\draw[white] (-.2,-.7) rectangle (12.2,2.7);
\draw (1,-.5) rectangle (3,2.5);
\draw (9,-.5) rectangle (11,2.5);
\draw (0,0) -- (1,0);
\draw (11,2) -- (12,2);
\begin{scope}[rounded corners = 1mm]
\draw[->] (3,0) -- (6,0) -- (6.4,.5) -- (6,1) -- (5.6,1.5) -- (6,2) -- (9,2);
\draw[->] (3,1) -- (3.4,1.5) -- (3,2);
\draw[->] (9,0) -- (8.6,.5) -- (9,1);
\draw[->] (5,2) arc (90:450:.4cm and .5cm);
\draw[->] (4,2) arc (90:-270:.4cm and .5cm);
\draw[->] (7,1) arc (90:-270:.4cm and .5cm);
\draw[->] (8,1) arc (90:450:.4cm and .5cm);
\end{scope}
\draw[densely dashed] (1,0) -- (3,0);
\draw[densely dashed, rounded corners =1mm] (3,1) -- (2.6,1.5) -- (3,2);
\draw[densely dashed] (9,2) -- (11,2);
\draw[densely dashed, rounded corners =1mm] (9,1) -- (9.4,.5) -- (9,0);

\end{tikzpicture}

\\
\hline
4A$_{pe}$ & $c-4$ & \tiny{$-+$} & 
% Case 4A knot after
\begin{tikzpicture}[scale=.4]
\draw[white] (-.2,-.7) rectangle (8.2,2.7);
\draw (1,-.5) rectangle (3,2.5);
\draw (2,1) node{$R$};
\draw (6,1) node[rotate = 180]{$\overline{R}$};
\draw (5,-.5) rectangle (7,2.5);
\draw (0,0) -- (1,0);
\draw (7,2) -- (8,2);
\begin{scope}[rounded corners = 1mm]
	\draw (3,0) -- (4,0) -- (5,1);
	\draw (3,1) -- (3.3,1.3);
	\draw (3.7,1.7) -- (4,2) -- (5,2);
	\draw (3,2) -- (4.3,.7);
	\draw (4.7,.3) -- (5,0);
\end{scope}

\draw[->] (3.5,1.5) -- (3.1,1.9);
\draw[->] (3.7,1.7) -- (3.9,1.9);
\draw[->] (4.5,.5) -- (4.9,.9);
\draw[->] (4.3,.7) -- (4.1,.9);

\end{tikzpicture}
		&
% Case 4A Seifert state after
 \begin{tikzpicture}[scale=.4]
\draw[white] (-.2,-.7) rectangle (8.2,2.7);
\draw (1,-.5) rectangle (3,2.5);
\draw (5,-.5) rectangle (7,2.5);
\draw (0,0) -- (1,0);
\draw (7,2) -- (8,2);
\begin{scope}[rounded corners = 1mm]
	\draw[->] (3,1) -- (3.4,1.5) -- (3,2);
	\draw[->] (5,0) -- (4.6,.5) -- (5,1);
	\draw[->] (3,0) -- (4,0) -- (4.4,.5) -- (4,1) -- (3.6,1.5) -- (4,2) -- (5,2);
\end{scope}
\draw[densely dashed] (1,0) -- (3,0);
\draw[densely dashed, rounded corners=1mm] (3,2) -- (2.6,1.5) -- (3,1);
\draw[densely dashed] (5,2) -- (7,2);
\draw[densely dashed, rounded corners=1mm] (5,1) -- (5.4,.5) -- (5,0);

\end{tikzpicture}
 \\
\hline
\hline
4B$_{pe}$ & $c$ &\tiny{$--++-+--++$} &
% Case 4B knot before
\begin{tikzpicture}[scale=.4]
\draw[white] (-.2,-.7) rectangle (12.2,2.7);
\draw (1,-.5) rectangle (3,2.5);
\draw (2,1) node{$R$};
\draw (10,1) node[rotate = 180]{$\overline{R}$};
\draw (9,-.5) rectangle (11,2.5);
\draw (0,0) -- (1,0);
\draw (11,2) -- (12,2);
\begin{scope}[rounded corners = 1mm]
	\draw (3,0) -- (4.3,1.3);
	\draw (4.7,1.7) -- (5,2) -- (6.3,.7);
	\draw (6.7,.3) -- (7,0) -- (8.3,1.3);
	\draw (8.7,1.7) -- (9,2);
	\draw (3,1) -- (3.3,.7);
	\draw (3.7,.3) -- (4,0) -- (6,0) -- (7,1) -- (7.3,.7);
	\draw (7.7,.3) -- (8,0) -- (9,0);
	\draw (3,2) -- (4,2) -- (5,1) -- (5.3,1.3);
	\draw (5.7,1.7) -- (6,2) -- (8,2) -- (9,1);
\end{scope}
\draw[->] (3.5,.5) -- (3.1,.1);
\draw[->] (3.7,.3) -- (3.9,.1);
\draw[->] (4.5,1.5) -- (4.9,1.1);
\draw[->] (4.3,1.3) -- (4.1,1.1);
\draw[->] (5.5,1.5) -- (5.1,1.9);
\draw[->] (5.7,1.7) -- (5.9,1.9);
\draw[->] (6.5,.5) --(6.9,.9);
\draw[->] (6.3,.7) -- (6.1,.9);
\draw[->] (7.5,.5) -- (7.1,.1);
\draw[->] (7.7,.3) -- (7.9,.1);
\draw[->] (8.5,1.5) -- (8.9,1.1);
\draw[->] (8.3,1.3) -- (8.1,1.1);

\end{tikzpicture}
& 
% Case 4B Seifert state before
\begin{tikzpicture}[scale=.4]
\draw[white] (-.2,-.7) rectangle (12.2,2.7);
\draw (1,-.5) rectangle (3,2.5);
\draw (9,-.5) rectangle (11,2.5);
\draw (0,0) -- (1,0);
\draw (11,2) -- (12,2);
\begin{scope}[rounded corners = 1mm]
\draw[->] (3,2) -- (4,2) -- (4.4,1.5) -- (4,1) -- (3.6,.5) -- (4,0) -- (6,0) -- (6.4,.5) -- (6,1) -- (5.6,1.5) -- (6,2) -- (8,2) -- (8.4,1.5) -- (8,1) -- (7.6,.5) -- (8,0) -- (9,0);
\draw[->] (3,1) -- (3.4,.5) -- (3,0);
\draw[->] (9,2) -- (8.6,1.5) -- (9,1);
\draw[->] (5,2) arc (90:450:.4cm and .5cm);
\draw[->] (7,1) arc (90:-270:.4cm and .5cm);
\end{scope}
\draw[densely dashed] (1,0) -- (3,2);
\draw[densely dashed, rounded corners =1mm] (3,0) -- (2.6,.5) -- (3,1);
\draw[densely dashed] (9,0) -- (11,2);
\draw[densely dashed, rounded corners =1mm] (9,1) -- (9.4,1.5) -- (9,2);

\end{tikzpicture}

\\
\hline
4B$_{pe}$ & $c-4$ & \tiny{$--++$} & 
% Case 4B knot after
\begin{tikzpicture}[scale=.4]
\draw[white] (-.2,-.7) rectangle (8.2,2.7);
\draw (1,-.5) rectangle (3,2.5);
\draw (2,1) node{$R$};
\draw (6,1) node[rotate = 180]{$\overline{R}$};
\draw (5,-.5) rectangle (7,2.5);
\draw (0,0) -- (1,0);
\draw (7,2) -- (8,2);
\begin{scope}[rounded corners = 1mm]
	\draw (3,0) -- (4.3,1.3);
	\draw (4.7,1.7) -- (5,2);
	\draw (3,1) -- (3.3,.7);
	\draw (3,2) -- (4,2) -- (5,1);
	\draw (3.7,.3) -- (4,0) -- (5,0);
\end{scope}

\draw[->] (3.5,.5) -- (3.1,.1);
\draw[->] (3.7,.3) -- (3.9,.1);
\draw[->] (4.5,1.5) -- (4.9,1.1);
\draw[->] (4.3,1.3) -- (4.1,1.1);

\end{tikzpicture}
		&
% Case 4B Seifert state after
 \begin{tikzpicture}[scale=.4]
\draw[white] (-.2,-.7) rectangle (8.2,2.7);
\draw (1,-.5) rectangle (3,2.5);
\draw (5,-.5) rectangle (7,2.5);
\draw (0,0) -- (1,0);
\draw (7,2) -- (8,2);
\begin{scope}[rounded corners = 1mm]
	\draw[->] (3,1) -- (3.4,.5) -- (3,0);
	\draw[->] (5,2) -- (4.6,1.5) -- (5,1);
	\draw[->] (3,2) -- (4,2) -- (4.4,1.5) -- (4,1) -- (3.6,.5) -- (4,0) -- (5,0);
\end{scope}
\draw[densely dashed] (1,0) -- (3,2);
\draw[densely dashed, rounded corners=1mm] (3,0) -- (2.6,.5) -- (3,1);
\draw[densely dashed] (5,0) -- (7,2);
\draw[densely dashed, rounded corners=1mm] (5,1) -- (5.4,1.5) -- (5,2);

\end{tikzpicture}
\\
\hline
\end{tabular}
\caption{Alternating diagrams and Seifert states corresponding to the even palindromic cases in the proof of Theorem \ref{thm:Seifertrecursionpalindrome}.}
\label{tab:SeifertPalindromeEven}
\end{table}

%-----------------------------------------------------TABLE c=6 colored NEW --------------------------------------------------------------------

\begin{table}
\begin{tabular}{|c|c||c|c|c|}
\hline
Case & Crossing & String & Alternating Diagram & Seifert state \\
 & Number & & & \\
\hline
\hline
1$_{po}$ & $c$ & \tiny{$+-{}-+$} & 
% Case 1 knot before
\begin{tikzpicture}[scale = .4]
\draw[white] (-.2,-.7) -- (9.2,2.7);
\draw (0,0) -- (1,0);
\draw (1,-.5) rectangle (3,2.5);
\draw (6,-.5) rectangle (8,2.5);
\draw (2,1) node{$R$};
\draw (7,1) node{$\reflectbox{R}$};
\draw (8,0) -- (9,0);
\begin{scope}[rounded corners = 1mm]
	\draw (3,0) -- (4,1) -- (4.3,.7);
	\draw (4.7,.3) -- (5,0) -- (6,1);
	\draw (3,1) -- (3.3,.7);
	\draw (3.7,.3) -- (4,0) -- (5,1) --(5.3,.7);
	\draw (5.7,.3) -- (6,0);
	\draw (3,2) -- (6,2);
\end{scope}
\draw[->] (3.5,.5) -- (3.9,.9);
\draw[->] (3.7,.3) -- (3.9,.1);
\draw[->] (4.5,.5) -- (4.9,.9);
\draw[->] (4.7,.3) -- (4.9,.1);
\draw[->] (5.5,.5) -- (5.9,.9);
\draw[->] (5.7,.3) -- (5.9,.1);
\draw[->] (4,2) -- (3.2,2);

\end{tikzpicture}
 & 
 % Case 1 Seifert state before
 \begin{tikzpicture}[scale = .4]
\draw[white] (-.2,-.7) -- (9.2,2.7);
\draw (0,0) -- (1,0);
\draw (1,-.5) rectangle (3,2.5);
\draw (6,-.5) rectangle (8,2.5);
\draw (8,0) -- (9,0);
\begin{scope}[rounded corners = 1mm]
	\draw[->] (3,0) -- (3.5,.4) -- (4,0) -- (4.5,.4) -- (5,0) -- (5.5,.4) -- (6,0);
	\draw[->] (3,1) -- (3.5,.6) -- (4,1) -- (4.5,.6) -- (5,1) -- (5.5,.6) -- (6,1);
	\draw[->] (6,2) -- (3,2);
	\draw[densely dashed] (1,0) -- (3,0);
	\draw[densely dashed] (3,1) -- (2.6,1.5) -- (3,2);
	\draw[densely dashed] (6,1) -- (6.4,1.5) -- (6,2);
	\draw[densely dashed] (6,0) -- (8,0);
\end{scope}

 \end{tikzpicture}
 
  \\
\hline
1$_{po}$ & $c-1$ & \tiny{$+$} & 
% Case 1 knot after
\begin{tikzpicture}[scale = .4]
\draw[white] (-.2,-.7) -- (7.2,2.7);
\draw (0,0) -- (1,0);
\draw (1,-.5) rectangle (3,2.5);
\draw (4,-.5) rectangle (6,2.5);
\draw (2,1) node{$R$};
\draw (5,1) node{$\reflectbox{R}$};
\draw (6,0) -- (7,0);
\draw (3,0) -- (4,1);
\draw (3,1) -- (3.3,.7);
\draw (3.7,.3) -- (4,0);
\draw (3,2) -- (4,2);
\draw[->] (3.5,.5) -- (3.9,.9);
\draw[->] (3.7,.3) -- (3.9,.1);
\draw[->] (4,2) -- (3.2,2);

\end{tikzpicture}
	&
 % Case 1 Seifert state after
 \begin{tikzpicture}[scale = .4]
\draw[white] (-.2,-.7) -- (7.2,2.7);
\draw (0,0) -- (1,0);
\draw (1,-.5) rectangle (3,2.5);
\draw (4,-.5) rectangle (6,2.5);
\draw (6,0) -- (7,0);
\begin{scope}[rounded corners = 1mm]
	\draw[->] (3,0) -- (3.5,.4) -- (4,0);
	\draw[->] (3,1) -- (3.5,.6) -- (4,1);
	\draw[->] (4,2) -- (3,2);
	\draw[densely dashed] (1,0) -- (3,0);
	\draw[densely dashed] (3,1) -- (2.6,1.5) -- (3,2);
	\draw[densely dashed] (4,1) -- (4.4,1.5) -- (4,2);
	\draw[densely dashed] (4,0) -- (6,0);
\end{scope}

 \end{tikzpicture}	

 \\
\hline
\hline
2$_{po}$ & $c$ & \tiny{$++-++$} & 
% Case 2 knot before
\begin{tikzpicture}[scale = .4]
\draw[white] (-.2,-.7) -- (9.2,2.7);
\draw (0,0) -- (1,0);
\draw (1,-.5) rectangle (3,2.5);
\draw (6,-.5) rectangle (8,2.5);
\draw (2,1) node{$R$};
\draw (7,1) node{$\reflectbox{R}$};
\draw (8,0) -- (9,0);
\begin{scope}[rounded corners = 1mm]
	\draw (3,2) -- (4,1) -- (4.3,1.3);
	\draw (4.7,1.7) -- (5,2) -- (6,1);
	\draw (3,1) -- (3.3,1.3);
	\draw (3.7,1.7) -- (4,2) -- (5,1) -- (5.3,1.3);
	\draw (5.7,1.7) -- (6,2);
	\draw (3,0) -- (6,0);

\end{scope}
\draw[->] (3.5,1.5) -- (3.9,1.1);
\draw[->] (3.7,1.7) -- (3.9,1.9);
\draw[->] (4.5,1.5) -- (4.9,1.1);
\draw[->] (4.7,1.7) -- (4.9,1.9);
\draw[->] (5.5,1.5) -- (5.9,1.1);
\draw[->] (5.7,1.7) -- (5.9,1.9);
\draw[->] (4,0) -- (3.2,0);

\end{tikzpicture}

 &
  % Case 2 Seifert state before
 \begin{tikzpicture}[scale = .4]
\draw[white] (-.2,-.7) -- (9.2,2.7);
\draw (0,0) -- (1,0);
\draw (1,-.5) rectangle (3,2.5);
\draw (6,-.5) rectangle (8,2.5);
\draw (8,0) -- (9,0);
\begin{scope}[rounded corners = 1mm]
	\draw[->] (3,2) -- (3.5,1.6) -- (4,2) -- (4.5,1.6) -- (5,2) -- (5.5,1.6) -- (6,2);
	\draw[->] (3,1) -- (3.5,1.4) -- (4,1) -- (4.5,1.4) -- (5,1) -- (5.5,1.4) -- (6,1);
	\draw[->] (6,0) -- (3,0);
	\draw[densely dashed] (1,0) -- (3,2);
	\draw[densely dashed] (3,1) -- (2.6,.5) -- (3,0);
	\draw[densely dashed] (6,1) -- (6.4,.5) -- (6,0);
	\draw[densely dashed] (6,2) -- (8,0);
\end{scope}

 \end{tikzpicture}
  \\
\hline
2$_{po}$ & $c-1$ & \tiny{$++$} & 
% Case 2 knot after
\begin{tikzpicture}[scale = .4]
\draw[white] (-.2,-.7) -- (7.2,2.7);
\draw (0,0) -- (1,0);
\draw (1,-.5) rectangle (3,2.5);
\draw (4,-.5) rectangle (6,2.5);
\draw (2,1) node{$R$};
\draw (5,1) node{$\reflectbox{R}$};
\draw (6,0) -- (7,0);
\draw (3,2) -- (4,1);
\draw (3,1) -- (3.3,1.3);
\draw (3.7,1.7) -- (4,2);
\draw (3,0) -- (4,0);
\draw[->] (3.5,1.5) -- (3.9,1.1);
\draw[->] (3.7,1.7) -- (3.9,1.9);
\draw[->] (4,0) -- (3.2,0);

\end{tikzpicture}
	&
 % Case 2 Seifert state after
 \begin{tikzpicture}[scale = .4]
\draw[white] (-.2,-.7) -- (7.2,2.7);
\draw (0,0) -- (1,0);
\draw (1,-.5) rectangle (3,2.5);
\draw (4,-.5) rectangle (6,2.5);
\draw (6,0) -- (7,0);
\begin{scope}[rounded corners = 1mm]
	\draw[->] (3,2) -- (3.5,1.6) -- (4,2);
	\draw[->] (3,1) -- (3.5,1.4) -- (4,1);
	\draw[->] (4,0) -- (3,0);
	\draw[densely dashed] (1,0) -- (3,2);
	\draw[densely dashed] (3,1) -- (2.6,.5) -- (3,0);
	\draw[densely dashed] (4,1) -- (4.4,.5) -- (4,0);
	\draw[densely dashed] (4,2) -- (6,0);
\end{scope}

 \end{tikzpicture}

\\
\hline
\hline
3A$_{po}$ & $c$ & \tiny{$-+-+-$} &
% Case 3A knot before
\begin{tikzpicture}[scale = .4]
\draw[white] (-.2,-.7) -- (11.2,2.7);
\draw (0,0) -- (1,0);
\draw (1,-.5) rectangle (3,2.5);
\draw (8,-.5) rectangle (10,2.5);
\draw (2,1) node{$R$};
\draw (9,1) node{$\reflectbox{R}$};
\draw (10,0) -- (11,0);
\begin{scope}[rounded corners = 1mm]
	\draw (3,0) -- (4,0) -- (5.3,1.3);
	\draw (5.7,1.7) -- (6,2) -- (7,2) -- (8,1);
	\draw (3,1) -- (3.3,1.3);
	\draw (3.7,1.7) -- (4,2) -- (5,2) -- (6.3,.7);
	\draw (6.7,.3) -- (7,0) -- (8,0);
	\draw (3,2) -- (4.3,.7);
	\draw (4.7,.3) -- (5,0) -- (6,0) -- (7.3,1.3);
	\draw (7.7,1.7) -- (8,2);

\end{scope}
\draw[->] (3.5,1.5) -- (3.1,1.9);
\draw[->] (3.7,1.7) -- (3.9,1.9);
\draw[->] (4.5,.5) -- (4.9,.9);
\draw[->] (4.3,.7) -- (4.1,.9);
\draw[->] (5.5,1.5) -- (5.9,1.1);
\draw[->] (5.7,1.7) -- (5.9,1.9);
\draw[->] (6.5,.5) -- (6.1,.1);
\draw[->] (6.7,.3) -- (6.9,.1);
\draw[->] (7.5,1.5) -- (7.9,1.1);
\draw[->] (7.3,1.3) -- (7.1,1.1);

\end{tikzpicture}

 &
  % Case 3A Seifert state before
 \begin{tikzpicture}[scale = .4]
\draw[white] (-.2,-.7) -- (11.2,2.7);
\draw (0,0) -- (1,0);
\draw (1,-.5) rectangle (3,2.5);
\draw (8,-.5) rectangle (10,2.5);
\draw (10,0) -- (11,0);
\begin{scope}[rounded corners = 1mm]
	\draw[->] (3,1) -- (3.4,1.5) -- (3,2);
	\draw[->] (8,2) -- (7.6,1.5) -- (8,1);
	\draw[->] (5.5,0) -- (5,0) -- (4.6,.5) -- (5,1) -- (5.5,1.4) -- (6,1) -- (6.4,.5) -- (6,0) -- (5.5,0);
	\draw[->] (3,0) --(4,0) -- (4.4,.5) -- (4,1) -- (3.6,1.5) -- (4,2) -- (5,2) -- (5.5,1.6) -- (6,2) -- (7,2) -- (7.4,1.5) -- (7,1) -- (6.6,.5) -- (7,0) -- (8,0);
	\draw[densely dashed] (1,0) -- (3,0);
	\draw[densely dashed] (3,1) -- (2.6,1.5) -- (3,2);
	\draw[densely dashed] (8,1) -- (8.4,1.5) -- (8,2);
	\draw[densely dashed] (8,0) -- (10,0);
\end{scope}

 \end{tikzpicture}
 \\
\hline
3A$_{po}$ & $c-4$ &\tiny{$--$} & % Case 3A knot after
\begin{tikzpicture}[scale = .4]
\draw[white] (-.2,-.7) -- (7.2,2.7);
\draw (0,0) -- (1,0);
\draw (1,-.5) rectangle (3,2.5);
\draw (4,-.5) rectangle (6,2.5);
\draw (2,1) node{$R$};
\draw (5,1) node{$\reflectbox{R}$};
\draw (6,0) -- (7,0);
\draw (3,0) -- (4,1);
\draw (3,1) -- (3.3,.7);
\draw (3.7,.3) -- (4,0);
\draw (3,2) -- (4,2);
\draw[->] (3.5,.5) -- (3.9,.9);
\draw[->] (3.7,.3) -- (3.9,.1);
\draw[->] (4,2) -- (3.2,2);

\end{tikzpicture}
	&
 % Case 3A Seifert state after
 \begin{tikzpicture}[scale = .4]
\draw[white] (-.2,-.7) -- (7.2,2.7);
\draw (0,0) -- (1,0);
\draw (1,-.5) rectangle (3,2.5);
\draw (4,-.5) rectangle (6,2.5);
\draw (6,0) -- (7,0);
\begin{scope}[rounded corners = 1mm]
	\draw[->] (3,0) -- (3.5,.4) -- (4,0);
	\draw[->] (3,1) -- (3.5,.6) -- (4,1);
	\draw[->] (4,2) -- (3,2);
	\draw[densely dashed] (1,0) -- (3,0);
	\draw[densely dashed] (3,1) -- (2.6,1.5) -- (3,2);
	\draw[densely dashed] (4,1) -- (4.4,1.5) -- (4,2);
	\draw[densely dashed] (4,0) -- (6,0);
\end{scope}

 \end{tikzpicture}	 \\
\hline
\hline
3B$_{po}$ & $c$ & \tiny{$-{}-+-+-{}-$} & 
% Case 3B knot before
\begin{tikzpicture}[scale = .4]
\draw[white] (-.2,-.7) -- (11.2,2.7);
\draw (0,0) -- (1,0);
\draw (1,-.5) rectangle (3,2.5);
\draw (8,-.5) rectangle (10,2.5);
\draw (2,1) node{$R$};
\draw (9,1) node{$\reflectbox{R}$};
\draw (10,0) -- (11,0);
\begin{scope}[rounded corners = 1mm]
	\draw (3,0) -- (4,1) -- (4.3,.7);
	\draw (4.7,.3) -- (5,0) -- (6,0) -- (7,1) -- (7.3,.7);
	\draw (7.7,.3) -- (8,0);
	\draw (3,1) -- (3.3,.7);
	\draw (3.7,.3) -- (4,0) -- (5.3,1.3);
	\draw (5.7,1.7) -- (6,2) -- (8,2);
	\draw (3,2) -- (5,2) -- (6.3,.7);
	\draw (6.7,.3) -- (7,0) -- (8,1);
\end{scope}
\draw[->] (3.5,.5) -- (3.1,.1);
\draw[->] (3.7,.3) -- (3.9,.1);
\draw[->] (4.5,.5) -- (4.9,.9);
\draw[->] (4.3,.7) -- (4.1,.9);
\draw[->] (5.5,1.5) -- (5.9,1.1);
\draw[->] (5.7,1.7) -- (5.9,1.9);
\draw[->] (6.5,.5) -- (6.1,.1);
\draw[->] (6.7,.3) -- (6.9,.1);
\draw[->] (7.5,.5) -- (7.9,.9);
\draw[->] (7.3,.7) -- (7.1,.9);

\end{tikzpicture}

 &
  % Case 3B Seifert state before
 \begin{tikzpicture}[scale = .4]
\draw[white] (-.2,-.7) -- (11.2,2.7);
\draw (0,0) -- (1,0);
\draw (1,-.5) rectangle (3,2.5);
\draw (8,-.5) rectangle (10,2.5);
\draw (10,0) -- (11,0);
\begin{scope}[rounded corners = 1mm]
	\draw[->] (3,2) -- (5,2) -- (5.5,1.6) -- (6,2) -- (8,2);
	\draw[->] (3,1) -- (3.4,.5) -- (3,0);
	\draw[->] (8,0) -- (7.6,.5) -- (8,1);
	\draw[->] (4,1) arc (90:450:.4cm and .5cm);
	\draw[->] (7,1) arc (90:450:.4cm and .5cm);
	\draw[->] (5.5,0) -- (5,0) -- (4.6,.5) -- (5,1) --(5.5,1.4) -- (6,1) -- (6.4,.5) -- (6,0) -- (5.5,0);
	\draw[densely dashed] (1,0) -- (3,2);
	\draw[densely dashed] (3,1) -- (2.6,.5) -- (3,0);
	\draw[densely dashed] (8,1) -- (8.4,.5) -- (8,0);
	\draw[densely dashed] (8,2) -- (10,0);
\end{scope}

 \end{tikzpicture}
 \\
\hline
3B$_{po}$ & $c-4$ & \tiny{$-$} & 
% Case 3b knot after
\begin{tikzpicture}[scale = .4]
\draw[white] (-.2,-.7) -- (7.2,2.7);
\draw (0,0) -- (1,0);
\draw (1,-.5) rectangle (3,2.5);
\draw (4,-.5) rectangle (6,2.5);
\draw (2,1) node{$R$};
\draw (5,1) node{$\reflectbox{R}$};
\draw (6,0) -- (7,0);
\draw (3,2) -- (4,1);
\draw (3,1) -- (3.3,1.3);
\draw (3.7,1.7) -- (4,2);
\draw (3,0) -- (4,0);
\draw[->] (3.5,1.5) -- (3.9,1.1);
\draw[->] (3.7,1.7) -- (3.9,1.9);
\draw[->] (4,0) -- (3.2,0);

\end{tikzpicture}
	&
 % Case 3b Seifert state after
 \begin{tikzpicture}[scale = .4]
\draw[white] (-.2,-.7) -- (7.2,2.7);
\draw (0,0) -- (1,0);
\draw (1,-.5) rectangle (3,2.5);
\draw (4,-.5) rectangle (6,2.5);
\draw (6,0) -- (7,0);
\begin{scope}[rounded corners = 1mm]
	\draw[->] (3,2) -- (3.5,1.6) -- (4,2);
	\draw[->] (3,1) -- (3.5,1.4) -- (4,1);
	\draw[->] (4,0) -- (3,0);
	\draw[densely dashed] (1,0) -- (3,2);
	\draw[densely dashed] (3,1) -- (2.6,.5) -- (3,0);
	\draw[densely dashed] (4,1) -- (4.4,.5) -- (4,0);
	\draw[densely dashed] (4,2) -- (6,0);
\end{scope}

 \end{tikzpicture}

\\
\hline
\hline
4A$_{po}$ & $c$ & \tiny{$-++-{}-++-$} &
% Case 3A knot before
\begin{tikzpicture}[scale = .4]
\draw[white] (-.2,-.7) -- (11.2,2.7);
\draw (0,0) -- (1,0);
\draw (1,-.5) rectangle (3,2.5);
\draw (8,-.5) rectangle (10,2.5);
\draw (2,1) node{$R$};
\draw (9,1) node{$\reflectbox{R}$};
\draw (10,0) -- (11,0);
\begin{scope}[rounded corners = 1mm]
	\draw (3,0) --(5,0) -- (6.3,1.3);
	\draw (6.7,1.7) -- (7,2) --(8,1);
	\draw (3,1) -- (3.3,1.3);
	\draw (3.7,1.7) -- (4,2) -- (5.3,.7);
	\draw (5.7,.3) -- (6,0) -- (8,0);
	\draw (3,2) -- (4,1) -- (4.3,1.3);
	\draw (4.7,1.7) -- (5,2) -- (6,2) -- (7,1) -- (7.3,1.3);
	\draw (7.7,1.7) -- (8,2);
\end{scope}
\draw[->] (3.5,1.5) -- (3.1,1.9);
\draw[->] (3.7,1.7) -- (3.9,1.9);
\draw[->] (4.5,1.5) -- (4.9,1.1);
\draw[->] (4.3,1.3) -- (4.1,1.1);
\draw[->] (5.5,.5) -- (5.9,.9);
\draw[->] (5.7,.3) -- (5.9,.1);
\draw[->] (6.5,1.5) -- (6.1,1.9);
\draw[->] (6.7,1.7) -- (6.9,1.9);
\draw[->] (7.5,1.5) -- (7.9,1.1);
\draw[->] (7.3,1.3) -- (7.1,1.1);

\end{tikzpicture}

 &
  % Case 4A Seifert state before
 \begin{tikzpicture}[scale = .4]
\draw[white] (-.2,-.7) -- (11.2,2.7);
\draw (0,0) -- (1,0);
\draw (1,-.5) rectangle (3,2.5);
\draw (8,-.5) rectangle (10,2.5);
\draw (10,0) -- (11,0);
\begin{scope}[rounded corners = 1mm]
	\draw[->] (3,0) -- (5,0) -- (5.5,0.4) -- (6,0) -- (8,0);
	\draw[->] (3,1) -- (3.4,1.5) -- (3,2);
	\draw[->] (8,2) -- (7.6,1.5) -- (8,1);
	\draw[->] (5.5,2) -- (5,2) -- (4.6,1.5) -- (5,1) -- (5.5,.6) -- (6,1) -- (6.4,1.5) -- (6,2) -- (5.5,2);
	\draw[->] (4,2) arc (90:-270:.4cm and .5cm);
	\draw[->] (7,2) arc (90:-270:.4cm and .5cm);
	\draw[densely dashed] (1,0) -- (3,0);
	\draw[densely dashed] (3,1) -- (2.6,1.5) -- (3,2);
	\draw[densely dashed] (8,1) -- (8.4,1.5) -- (8,2);
	\draw[densely dashed] (8,0) -- (10,0);
\end{scope}

 \end{tikzpicture}

 \\
\hline
4A$_{po}$ & $c-4$ & \tiny{$--$} & 
% Case 4A knot after
\begin{tikzpicture}[scale = .4]
\draw[white] (-.2,-.7) -- (7.2,2.7);
\draw (0,0) -- (1,0);
\draw (1,-.5) rectangle (3,2.5);
\draw (4,-.5) rectangle (6,2.5);
\draw (2,1) node{$R$};
\draw (5,1) node{$\reflectbox{R}$};
\draw (6,0) -- (7,0);
\draw (3,0) -- (4,1);
\draw (3,1) -- (3.3,.7);
\draw (3.7,.3) -- (4,0);
\draw (3,2) -- (4,2);
\draw[->] (3.5,.5) -- (3.9,.9);
\draw[->] (3.7,.3) -- (3.9,.1);
\draw[->] (4,2) -- (3.2,2);

\end{tikzpicture}
	&
 % Case 4A Seifert state after
 \begin{tikzpicture}[scale = .4]
\draw[white] (-.2,-.7) -- (7.2,2.7);
\draw (0,0) -- (1,0);
\draw (1,-.5) rectangle (3,2.5);
\draw (4,-.5) rectangle (6,2.5);
\draw (6,0) -- (7,0);
\begin{scope}[rounded corners = 1mm]
	\draw[->] (3,0) -- (3.5,.4) -- (4,0);
	\draw[->] (3,1) -- (3.5,.6) -- (4,1);
	\draw[->] (4,2) -- (3,2);
	\draw[densely dashed] (1,0) -- (3,0);
	\draw[densely dashed] (3,1) -- (2.6,1.5) -- (3,2);
	\draw[densely dashed] (4,1) -- (4.4,1.5) -- (4,2);
	\draw[densely dashed] (4,0) -- (6,0);
\end{scope}

 \end{tikzpicture}\\
\hline
\hline
4B$_{po}$ & $c$ & \tiny{$-{}-++-{}-++-{}-$} &
% Case 4B knot before
\begin{tikzpicture}[scale = .4]
\draw[white] (-.2,-.7) -- (11.2,2.7);
\draw (0,0) -- (1,0);
\draw (1,-.5) rectangle (3,2.5);
\draw (8,-.5) rectangle (10,2.5);
\draw (2,1) node{$R$};
\draw (9,1) node{$\reflectbox{R}$};
\draw (10,0) -- (11,0);
\begin{scope}[rounded corners = 1mm]
	\draw (3,0) -- (4.3,1.3);
	\draw (4.7,1.7) -- (5,2) -- (6,2) -- (7.3,.7);
	\draw (7.7,.3) -- (8,0);
	\draw (3,1) -- (3.3,.7);
	\draw (3.7,.3) -- (4,0) -- (5,0) -- (6.3,1.3);
	\draw (6.7,1.7) -- (7,2) -- (8,2);
	\draw (3,2) -- (4,2) -- (5.3,.7);
	\draw (5.7,.3) -- (6,0) -- (7,0) -- (8,1);
\end{scope}
\draw[->] (3.5,.5) -- (3.1,.1);
\draw[->] (3.7,.3) -- (3.9,.1);
\draw[->] (4.5,1.5) -- (4.9,1.1);
\draw[->] (4.3,1.3) -- (4.1,1.1);
\draw[->] (5.5,.5) -- (5.9,.9);
\draw[->] (5.7,.3) -- (5.9,.1);
\draw[->] (6.5,1.5) -- (6.1,1.9);
\draw[->] (6.7,1.7) -- (6.9,1.9);
\draw[->] (7.5,.5) -- (7.9,.9);
\draw[->] (7.3,.7) -- (7.1,.9);

\end{tikzpicture}

 &
  % Case 4B Seifert state before
 \begin{tikzpicture}[scale = .4]
\draw[white] (-.2,-.7) -- (11.2,2.7);
\draw (0,0) -- (1,0);
\draw (1,-.5) rectangle (3,2.5);
\draw (8,-.5) rectangle (10,2.5);
\draw (10,0) -- (11,0);
\begin{scope}[rounded corners = 1mm]
	\draw[->] (3,2) -- (4,2) -- (4.4,1.5) -- (4,1) -- (3.6,.5) -- (4,0) -- (5,0) -- (5.5,.4) -- (6,0) --(7,0) -- (7.4,.5) -- (7,1) -- (6.6,1.5) -- (7,2) -- (8,2);
	\draw[->] (3,1) -- (3.4,.5) -- (3,0);
	\draw[->] (8,0) -- (7.6,.5) -- (8,1);
	\draw[->] (5.5,2) -- (5,2) -- (4.6,1.5) -- (5,1) --(5.5,.6) -- (6,1) -- (6.4,1.5) -- (6,2) -- (5.5,2);
	\draw[densely dashed] (1,0) -- (3,2);
	\draw[densely dashed] (3,1) -- (2.6,.5) -- (3,0);
	\draw[densely dashed] (8,1) -- (8.4,.5) -- (8,0);
	\draw[densely dashed] (8,2) -- (10,0);
\end{scope}

 \end{tikzpicture}

 \\
\hline
4B$_{po}$ & $c-4$ & \tiny{$-$} &% Case 4b knot after
\begin{tikzpicture}[scale = .4]
\draw[white] (-.2,-.7) -- (7.2,2.7);
\draw (0,0) -- (1,0);
\draw (1,-.5) rectangle (3,2.5);
\draw (4,-.5) rectangle (6,2.5);
\draw (2,1) node{$R$};
\draw (5,1) node{$\reflectbox{R}$};
\draw (6,0) -- (7,0);
\draw (3,2) -- (4,1);
\draw (3,1) -- (3.3,1.3);
\draw (3.7,1.7) -- (4,2);
\draw (3,0) -- (4,0);
\draw[->] (3.5,1.5) -- (3.9,1.1);
\draw[->] (3.7,1.7) -- (3.9,1.9);
\draw[->] (4,0) -- (3.2,0);

\end{tikzpicture}
	&
 % Case 4b Seifert state after
 \begin{tikzpicture}[scale = .4]
\draw[white] (-.2,-.7) -- (7.2,2.7);
\draw (0,0) -- (1,0);
\draw (1,-.5) rectangle (3,2.5);
\draw (4,-.5) rectangle (6,2.5);
\draw (6,0) -- (7,0);
\begin{scope}[rounded corners = 1mm]
	\draw[->] (3,2) -- (3.5,1.6) -- (4,2);
	\draw[->] (3,1) -- (3.5,1.4) -- (4,1);
	\draw[->] (4,0) -- (3,0);
	\draw[densely dashed] (1,0) -- (3,2);
	\draw[densely dashed] (3,1) -- (2.6,.5) -- (3,0);
	\draw[densely dashed] (4,1) -- (4.4,.5) -- (4,0);
	\draw[densely dashed] (4,2) -- (6,0);
\end{scope}
\end{tikzpicture}
 \\
\hline
\end{tabular}
\caption{Alternating diagrams and Seifert states corresponding to the odd palindromic cases in the proof of Theorem \ref{thm:Seifertrecursionpalindrome}.}
\label{tab:SeifertPalindromeOdd}
\end{table}

\section{Seifert circles and average genus}
\label{sec:formulas}

In Section \ref{sec:recursions}, we find recursive formulas for the total number of Seifert circles $s(c)$ and $s_p(c)$ coming from the alternating diagrams associated to words in $T(c)$ and $T_p(c)$, respectively. In this section, we find closed formulas for $s(c)$ and $s_p(c)$, and then use those formulas to prove Theorem \ref{thm:mainformula}.

The total number $s(c)$ of Seifert circles in the alternating diagrams coming from words in $T(c)$ is given by the following theorem.
\begin{theorem}
\label{thm:s(c)}
Let $c\geq 3$. The number  $s(c)$ of Seifert circles in the alternating diagrams with crossing number $c$ coming from words in $T(c)$ can be expressed as
\[ s(c) = \frac{(3c+5)2^{c-3}+(-1)^c (5-3c)}{9}.\]
\end{theorem}
\begin{proof}
Recall that $s(c)$  satisfies the recurrence relation $s(c) = s(c-1) + 2s(c-2) + 3t(c-2)$ with initial conditions $s(3)=2$ and  $s(4)=3$ and that $3t(c-2) = 2^{c-4}-(-1)^{c-4}$.

Proceed by induction. The base cases of $s(3)=2$ and $s(4)=3$ can be shown by direct computation. The recurrence relation is satisfied because
\begin{align*}
& s(c-1) + 2s(c-2) + 3t(c-2)\\
 = & \; \frac{[3(c-1)+5]2^{(c-1)-3}+(-1)^{c-1}[5-3(c-1)]}{9} \\
  & \; + 2\left(\frac{[3(c-2)+5]2^{(c-2)-3} + (-1)^{c-2}[5-3(c-2)]}{9}\right) + 2^{c-4} - (-1)^{c-4} \\
 = & \; \frac{(3c+2)2^{c-4} + (-1)^c(3c-8)+(3c-1)2^{c-4} + (-1)^c(22-6c) + 9\cdot 2^{c-4} - 9 (-1)^c}{9}\\
  = & \; \frac{(6c+10)2^{c-4} +(-1)^c[(3c-8) +(22-6c) -9]}{9}\\
  = & \; \frac{(3c+5)2^{c-3}+(-1)^c (5-3c)}{9}.
\end{align*}
\end{proof}

The total number $s_p(c)$ of Seifert circles in the alternating diagrams coming from words of palindromic type in $T_p(c)$ is given by the following theorem.
\begin{theorem}
\label{thm:sp(c)}
Let $c\geq 3$. The number $s_p(c)$ of Seifert circles in the alternating diagrams coming from words of palindromic type in $T_p(c)$ can be expressed as
\[s_p(c) = \begin{cases}\displaystyle
\frac{(3c+1)2^{(c-3)/2} + (-1)^{(c-1)/2}(1-3c)}{9} & \text{if $c$ is odd,}\\
\displaystyle
\frac{(3c+4)2^{(c-4)/2} + (-1)^{(c-2)/2}(1-3c)}{9} & \text{if $c$ is even.}
\end{cases}\]
\end{theorem}

\begin{proof}
Recall that $s_p(c)$  satisfies the recurrence relation $s_p(c) = s_p(c-2) + 2s_p(c-4) + 6t_p(c-4)$ with initial conditions $s_p(3)=2,$ $s_p(4)=3$, $s_p(5)=2$, and $s_p(6) = 3$ .

Proceed by induction. One may verify the initial conditions by direct computation. Since the recursion relation for $s_p(c)$ either involves only odd indexed terms or only even indexed terms, we handle each case separately. Suppose $c$ is odd. Then Proposition \ref{prop:numberpalindromic} implies that $ t_p(c-4) =  J(\frac{c-5}{2}) = \frac{2^{(c-5)/2} - (-1)^{(c-5)/2})}{3}$. Thus
\begin{align*}
& \;s_p(c-2) + 2s_p(c-4) + 6t_p(c-4)\\
= & \; \frac{ (3(c-2)+1) 2^{ ((c-2)-3)/2 }  + (-1)^{ ((c-2)-1)/2 } (1-3(c-2)) } { 9 }\\
 & \; + 2\left(\frac{(3(c-4)+1)2^{((c-4)-3)/2} + (-1)^{((c-4)-1)/2}(1-3(c-4))}{9}\right) + 6\left(\frac{2^{(c-5)/2} - (-1)^{(c-5)/2}}{3}\right)\\
 = &\; \frac{ (3c-5) 2^{(c-5)/2} + (-1)^{(c-3)/2}(7-3c)}{9}\\
 & \; + \frac{(3c-11) 2^{(c-5)/2} +(-1)^{(c-5)/2}(26-6c)}{9} + \frac{18 \cdot 2^{(c-5)/2} -(-1)^{(c-5)/2} \cdot 18}{9}\\
= & \; \frac{(6c+2)2^{(c-5)/2}  + (-1)^{(c-1)/2}((3c-7) + (26-6c) -18)}{9}\\
= & \; \frac{(3c+1)2^{(c-3)/2} + (-1)^{(c-1)/2}(1-3c)}{9}.
\end{align*}

Suppose $c$ is even. Then Proposition \ref{prop:numberpalindromic} implies $t_p(c-4) = J(\frac{c-6}{2})= \frac{2^{(c-6)/2} - (-1)^{(c-6)/2}}{3}$. Thus
\begin{align*}
& \; s_p(c-2) + 2s_p(c-4) + 6t_p(c-4)  \\
= & \; \frac{ (3(c-2)+4)2^{((c-2)-4)/2} + (-1)^{((c-2)-2)/2}(1-3(c-2))}{9}\\
& \; + 2\left( \frac{ (3(c-4)+4)2^{((c-4)-4)/2} + (-1)^{((c-4)-2)/2}(1-3(c-4))}{9} \right) + 6\left(\frac{2^{(c-6)/2} - (-1)^{(c-6)/2}}{3}\right)\\
= & \; \frac{(3c-2) 2^{(c-6)/2} + (-1)^{(c-4)/2}(7-3c)}{9} \\
& + \; \frac{(3c-8)2^{(c-6)/2} + (-1)^{(c-6)/2} (26-6c)}{9} + \frac{18\cdot 2^{(c-6)/2} - (-1)^{(c-6)/2}\cdot 18}{9}\\
= & \; \frac{ (6c+8)2^{(c-6)/2} + (-1)^{(c-2)/2} ((3c-7) + (26-6c) -18)}{9}\\
= & \; \frac{(3c+4)2^{(c-4)/2} + (-1)^{(c-2)/2}(1-3c)}{9}.
\end{align*}
\end{proof}

Although the proofs of Theorems \ref{thm:s(c)} and \ref{thm:sp(c)} are straightforward, finding the formulas for $s(c)$ and $s_p(c)$ involved combining several closed formulas found in the Online Encyclopedia of Integer Sequences \cite{OEIS}. We use the formulas for $|\mathcal{K}_c|$, $s(c)$, and $s_p(c)$ in Theorems \ref{thm:ernstsumners}, \ref{thm:s(c)}, and \ref{thm:sp(c)}, respectively to prove Theorem \ref{thm:mainformula}.
\begin{proof}[Proof of Theorem \ref{thm:mainformula}]
If $K$ is an alternating knot, then Murasugi \cite{Mur:genus} and Crowell \cite{Cro:genus} showed that its genus is $g(K) = \frac{1}{2}(1+c(K)-s(K))$  where $c(K)$ and $s(K)$ are the crossing number and number of components in the Seifert state of a reduced alternating diagram of $K$. Theorem \ref{thm:list} implies that 
\[\sum_{K\in\mathcal{K}_c} s(K) = \frac{1}{2}(s(c) + s_p(c)).\]
As in Equation (\ref{eq:avgenus}), the average genus $\overline{g}_c$ satisfies
\[
\overline{g}_c =  \frac{\sum_{K\in\mathcal{K}_c} g(K)}{|\mathcal{K}_c|}
=  \frac{\sum_{K\in\mathcal{K}_c} (1 + c - s(K))}{2|\mathcal{K}_c|}
=  \frac{1}{2} + \frac{c}{2} - \frac{s(c) + s_p(c)}{4|\mathcal{K}_c|}.\]

Theorems \ref{thm:ernstsumners}, \ref{thm:s(c)} and \ref{thm:sp(c)} contain expressions for $|\mathcal{K}_c|$, $s(c)$, and $s_p(c)$ that depend on $c$ mod $4$. If $c\equiv 0$ mod $4$, then 
\begin{align*}
\frac{s(c) + s_p(c)}{4|\mathcal{K}_c|} = & \;  \frac{(3c+5)2^{c-3}+(5-3c) + (3c+4)2^{(c-4)/2} +(3c-1)}{12(2^{c-3}+2^{(c-4)/2}}\\
= & \;  \frac{(3c+5)2^{c-3}+ (3c+5)2^{(c-4)/2}  -  2^{(c-4)/2} + 4}{12(2^{c-3}+2^{(c-4)/2})}\\
= & \; \frac{(3c+5)(2^{c-3}+2^{(c-4)/2})}{12(2^{c-3}+2^{(c-4)/2})} + \frac{4 - 2^{(c-4)/2}}{12(2^{c-3}+2^{(c-4)/2})}\\
= & \; \frac{c}{4} + \frac{5}{12} + \frac{4 - 2^{(c-4)/2}}{12(2^{c-3}+2^{(c-4)/2})}.
\end{align*}
When $c\equiv 0$ mod $4$, the average genus is
\[\overline{g}_c = \frac{c}{4} + \frac{1}{12} +  \frac{2^{(c-4)/2}-4 }{12(2^{c-3}+2^{(c-4)/2})}.\]
The cases where $c\equiv 1$, $2$, or $3$ mod $4$ are similar.

\end{proof}

\bibliographystyle{amsalpha}
%\bibliography{21AugBibliography}
%\end{document}

\end{document}